\documentclass[centertags,12pt]{article}
\usepackage{amssymb}
\usepackage{latexsym}
\usepackage{graphics}
\usepackage{graphicx}
\usepackage{enumerate}
\usepackage{amsmath}
\usepackage{amsthm}
\makeatletter
\renewcommand{\subsection}{\@startsection
{subsection}{2}{0mm}{\baselineskip}{-0.25cm}
{\normalfont\normalsize\em}}
\makeatother

 \usepackage{color} 
\usepackage{cancel}

\newcommand{\Sq}{{S(q)}}
\newcommand{\Rq}{{R(q)}}
\newcommand{\cSq}{{\mathcal{S}_q}}
\newcommand{\cRq}{{\mathcal{R}_q}}
\newcommand{\tSq}{{\tilde{S}(q)}}
\newcommand{\tRq}{{\tilde{R}(q)}}
\newcommand{\tcSq}{{\tilde{\mathcal{S}}_q}}
\newcommand{\tcRq}{{\tilde{\mathcal{R}}_q}}
\newcommand{\fq}{{\mathbb{F}_q}}
\newcommand{\f}{{\mathbb{F}}}
\newcommand{\cO}{{\mathcal{O}}}
\newcommand{\aut}{{\rm Aut}}

\newtheorem{theorem}{Theorem}
\newtheorem{proposition}[theorem]{Proposition}
\newtheorem{corollary}[theorem]{Corollary}
\newtheorem{lemma}[theorem]{Lemma}
\newtheorem{remark}[theorem]{Remark}
\title{On some Galois covers of the Suzuki and Ree curves}
\date{}
\author{M. Giulietti, M. Montanucci, L. Quoos, G. Zini}

\begin{document}

\maketitle

\abstract{We determine the full automorphism group of two recently constructed families $\tcSq$ and $\tcRq$ of maximal curves over finite fields. These curves are cyclic covers of the Suzuki and Ree curves, and are analogous to the Giulietti-Korchm\'aros cover of the Hermitian curve.
We show that $\tcSq$ is not Galois covered by the Hermitian curve maximal over $\mathbb{F}_{q^4}$, and $\tcRq$ is not Galois covered by the Hermitian curve maximal over $\mathbb{F}_{q^6}$.
Finally, we compute the genera of many Galois subcovers of $\tcSq$ and $\tcRq$; this provides new genera for maximal curves.
}

\section{Introduction}

For any prime power $\ell$, let $\mathbb F_\ell$ be the finite field with $\ell$ elements, and $\mathcal X$ be an $\mathbb F_\ell$-rational curve, i.e. a projective, absolutely irreducible, non-singular algebraic curve defined over $\mathbb F_\ell$.
The curve $\mathcal X$ is called $\mathbb F_\ell$-maximal if the number $|\mathcal X(\mathbb F_\ell)|$ of its $\mathbb F_\ell$-rational points attains the Hasse-Weil upper bound $\ell+1+2g\sqrt{\ell}$, where $g$ is the genus of $\mathcal X$.
Important examples of maximal curves over suitable finite fields are the so-called Deligne-Lusztig curves, namely the Hermitian, Suzuki, and Ree curves. All these curves have a large automorphism group when compared with their genus, they do not satisfy the classical Hurwitz bound $|{\rm Aut}(\mathcal X)| \leq 84(g-1)$ and have being intensively studied in the last decades. Determining subfields of their function fields as well the automorphism group of these subfields (see \cite{CO, GSX, GKT2006}), connections with class field theory (see \cite{L}) and applications in coding theory (see \cite{CZ, MST}) are some subjects of interest, see also  \cite{HKT} for more on algebraic curves and references therein.

Let $q_0$ be any prime power and $q=q_0^2$. Giulietti and Korchm\'aros \cite{GK2008} constructed the curve $\tilde{\mathcal H}_q$ which is $\mathbb F_{q^3}$-maximal and can be defined by the affine equations
$$
\tilde{\mathcal H}_q:\left\{
\begin{array}{ll}
t^{m}=x^q - x\\
x^{q_0+1}=y^{q_0}+y
\end{array}
\right.,
$$
where $m=q-q_0+1$ (see also \cite[Remark 2.1]{GGS}). Clearly, $\tilde{\mathcal H}_q$ is a Galois cover of the $\mathbb{F}_{q_0^2}$-maximal Hermitian curve $\mathcal H_{q_0}:x^{q_0+1}=y^{q_0}+y$ with automorphism group ${\rm Aut}(\mathcal H_{q_0}) \cong {\rm PGU}(3, q_0)$. The automorphism group ${\rm Aut}(\tilde{\mathcal H}_q)$ of $\tilde{\mathcal H}_q$ is defined over $\mathbb{F}_{q^3}$ and has a normal subgroup of index $d:=\gcd(3,m)$ which is isomorphic to ${\rm SU}(3,q_0)\times C_{m/d}$, where ${\rm SU}(3,q_0)$ is the special unitary group over $\mathbb F_{q}$ and $C_{m/d}$ is a cyclic group of order $m/d$.

Analogously, Skabelund \cite{Sk2016} constructed Galois covers of the Suzuki and Ree curves as follows.
Let $q_0=2^s$ with $s\geq1$ and $q=2q_0^2=2^{2s+1}$. The curve
$$
\tilde{\mathcal S}_q:\left\{
\begin{array}{ll}
t^{m}=x^q + x\\
y^q + y=x^{q_0}\left( x^q + x \right)
\end{array}
\right.,
$$
where $m=q-2q_0+1$, is $\mathbb F_{q^4}$-maximal. Clearly, $\tilde{\mathcal S}_q$ is a cyclic Galois cover of the Suzuki curve $$\mathcal S_q: y^q + y=x^{q_0}\left( x^q + x \right).$$

Now let $q_0=3^s$ with $s\geq1$ and $q=3q_0^2=3^{2s+1}$. The curve
$$
\tilde{\mathcal R}_q:\left\{
\begin{array}{ll}
t^{m}=x^q - x\\
z^q - z=x^{2q_0}\left( x^q - x \right)\\
y^q - y=x^{q_0}\left( x^q - x \right)
\end{array}
\right.,
$$
where $m=q-3q_0+1$, is $\mathbb F_{q^6}$-maximal. Clearly, $\tilde{\mathcal R}_q$ is a Galois cover of the Ree curve 
$$
\mathcal R_q:\left\{
\begin{array}{ll}
z^q - z=x^{2q_0}\left( x^q - x \right)\\
y^q - y=x^{q_0}\left( x^q - x \right)
\end{array}
\right..
$$
Also in \cite{Sk2016} the automorphism groups $S(q)$ of $\cSq$ and $R(q)$ of $\cRq$ were lifted to subgroups of the full automorphism groups ${\rm Aut}(\tilde{S}_q)$ and ${\rm Aut}(\tilde{R}_q)$ of $\tilde{\mathcal S}_q$ and $\tilde{\mathcal R}_q$, respectively. We show that the lifted groups are actually the full automorphism groups of $\tilde{\mathcal S}_q$ and $\tilde{\mathcal R}_q$. More specifically, we prove the following theorems.

\begin{theorem}\label{ThSuzuki}
The automorphism group of $\tilde{\mathcal S}_q$ is a direct product $\tilde{S}(q)\times C_m$, where $\tilde{S}(q)$ is isomorphic to the Suzuki group ${\rm Aut}(\mathcal S_q)$ and $C_m$ is a cyclic group of order $m=q-2q_0+1$.
\end{theorem}

\begin{theorem}\label{ThRee}
The automorphism group of $\tilde{\mathcal R}_q$ is a direct product $\tilde{R}(q)\times C_m$, where $\tilde{R}(q)$ is isomorphic to the Ree group ${\rm Aut}(\mathcal R_q)$ and $C_m$ is a cyclic group of order $m=q-3q_0+1$.
\end{theorem}

Theorems \ref{ThSuzuki} and \ref{ThRee} will be proved by means of results on finite automorphism groups.
In particular, we will use results on curves having automorphism groups for which the classical Hurwitz bound does not hold.

We compute the genus of the quotient curves $\tcSq/G$ or $\tcRq/G$ whenever $G=H\times C_n$, where $C_n\leq C_m$ and $H\leq \tSq$ or $\tRq$, respectively.
This is achieved by computing the ramification structure at the places of $\tcSq$ and $\tcRq$, and by using the known classification of subgroups of $\aut(\cSq)$ and $\aut(\cRq)$.

We also prove the following results.
\begin{theorem}\label{CoveredSuzuki}
For any $q$, $\tcSq$ is not Galois covered by $\mathcal{H}_{q^2}$.
\end{theorem}
\begin{theorem}\label{CoveredRee}
For any $q$, $\tcRq$ is not Galois covered by $\mathcal{H}_{q^3}$.
\end{theorem}

This paper is organized as follows.
Sections \ref{Sec:PreliminaryResultsSuzuki} and \ref{Sec:PreliminaryResultsRee} summarize the preliminary results on $\tilde{\mathcal S}_q$ and $\tilde{\mathcal R}_q$.
Sections \ref{Sec:Suzuki} and \ref{Sec:Ree} prove Theorems \ref{ThSuzuki} and \ref{ThRee}, respectively.
Section \ref{Sec:NonGaloisCovering} proves Theorems \ref{CoveredSuzuki} and \ref{CoveredRee}.
Section \ref{Sec:QuotientsSuzuki} provides the ramification structure of quotient curves $\tcSq/G$ under $\tcSq$, and the genus of $\tcSq/G$ whenever $G=H\times C_n$ with $H\leq\tSq$ and $C_n\leq C_m$. Section \ref{Sec:QuotientsRee} is analogous to Section \ref{Sec:QuotientsSuzuki} for the quotients of $\tcRq$.
Finally, Section \ref{Sec:NewGenera} shows new genera for maximal curves over $\mathbb F_{2^{12}}$, $\mathbb F_{2^{20}}$, and $\mathbb F_{3^{18}}$.

\section{Preliminary results on the curve $\tilde{\mathcal S}_q$}\label{Sec:PreliminaryResultsSuzuki}

For $s\geq1$ and $q=2q_0^2=2^{2s+1}$, the Suzuki curve $\mathcal S_q$ over $\fq$ is defined by the affine equation $Y^q+Y=X^{q_0}\left(X^q+X\right)$, has genus $q_0(q-1)$ and is maximal over $\mathbb F_{q^4}$. The automorphism group $S(q):={\rm Aut}(\mathcal S_q)$ of $\mathcal S_q$ is isomorphic to the Suzuki group $^2B_2(q)$. 
We state some other properties of $S(q)$ that will be used in the paper; see \cite{GKT2006} for more details.
\begin{itemize}
\item $S(q)$ has size $(q^2+1)q^2(q-1)$ and is a simple group.
\item $S(q)$ is generated by the stabilizer 
$$S(q)_{P_\infty}=\left\{\psi_{a,b,c}:(x,y)\mapsto (ax+b,a^{q_0+1}y+b^{q_0}x+c)\,|\,a,b,c\in\fq ,a\ne0\right \}$$ 
of the unique infinite place of $\cSq$, together with the involution $\phi:(x,y)\mapsto(\alpha/\beta,y/\beta)$, where $\alpha:=y^{2q_0}+x^{2q_0+1}$ and $\beta:=xy^{2q_0}+\alpha^{2q_0}$.
\item $S(q)$ has exactly two short orbits on $\mathcal S_q$. One is non-tame of size $q^2+1$, consisting of all $\mathbb F_q$-rational places; the other is tame of size $q^2(q-1)(q+2q_0+1)$, consisting of all $\mathbb F_{q^4}\setminus\mathbb F_q$-rational places. The group $S(q)$ acts $2$-transitively on its non-tame short orbit, and the stabilizer $S(q)_{P,Q}$ of  two distinct $\fq$-rational places $P$ and $Q$ is tame and cyclic.
\end{itemize}

Skabelund \cite[Sec. 3]{Sk2016} introduced the $\mathbb F_{q^4}$-maximal curve $\tilde{\mathcal S}_q$ defined over $\fq$ with affine equations
$$
\tilde{\mathcal S}_q:\left\{
\begin{array}{ll}
t^{m}=x^q + x\\
y^q + y=x^{q_0}\left( x^q + x \right)
\end{array}
\right.,
$$
where $m=q-2q_0+1$.

The curve $\tcSq$ is a degree-$m$ cyclic Galois cover of $\cSq$ and,
by the Riemann-Hurwitz formula, has genus $\tilde{g}=\frac{1}{2}\left(q^3-2q^2+q\right)$. The set of $\fq$-rational places of $\tcSq$ has size $q^2+1$. It consists of the places centered at the affine points of $\tcSq$ lying on the plane $t=0$, together with the $\fq$-rational infinite place. These places correspond exactly to the $\fq$-rational places of $S_q$.

The automorphism group ${\rm Aut}(\tcSq)$ of $\tcSq$ admits the following subgroups:
\begin{itemize}
\item A cyclic group $C_m$ generated by the automorphism $\gamma_\lambda:(x,y,t)\mapsto(x,y,\lambda t)$, where $\lambda\in\f_{q^4}$ is a primitive $m$-th root of unity.
\item A group $LS(q)$ lifted by $S(q)$ generated by the automorphisms $\tilde{\psi}_{a,b,c}$ ($a,b,c\in\fq$, $a\ne0$) together with an involution $\tilde{\phi}$.  Here, $\tilde{\psi}_{a,b,c}(x,y):=\psi_{a,b,c}(x,y)$ and $\tilde{\psi}_{a,b,c}(t):=\delta t$, where $\delta^m=a$.
Similarly, $\tilde{\phi}(x,y):=\phi(x,y)$, and $\tilde{\phi}(t):=t/\beta$ (see \cite[Section 3]{Sk2016}).
\end{itemize}

\section{Preliminary results on the curve $\tilde{\mathcal R}_q$}\label{Sec:PreliminaryResultsRee}

For $s\geq1$ and $q=3q_0^2=3^{2s+1}$, the Ree curve $\cRq$ over $\fq$ is defined by the affine equations
$$
\cRq:\left\{
\begin{array}{ll}
z^q - z=x^{2q_0}\left( x^q - x \right)\\
y^q - y=x^{q_0}\left( x^q - x \right)
\end{array}
\right.,
$$
has genus $\frac{3}{2}q_0(q-1)(q+q_0+1)$ and is maximal over $\f_{q^6}$.
The automorphism group $\Rq:={\rm Aut}(\cRq)$ is isomorphic to the simple Ree group $^2 G_2(q)$.
We state some other properties of $R(q)$ that will be used in the paper; see \cite{Pedersen1992} and \cite{Sk2016}.
\begin{itemize}
\item $R(q)$ has size $(q^3+1)q^3(q-1)$.
\item $R(q)$ is generated by the stabilizer
$$R(q)_{P_\infty} = \left\{\psi_{a,b,c,d}\,|\,a,b,c,d\in\fq,a\ne0\right\},$$
$$ \psi_{a,b,c,d}:(x,y,z)\mapsto(ax+b,a^{q_0+1}y+ab^{q_0}x+c, a^{2q_0+1}z - a^{q_0+1}b^{q_0}y + ab^{2q_0}x + d), $$
of the unique infinite place of $\cRq$, together with the involution $\phi:(x,y,z)\mapsto(w_6/w_8,w_{10}/w_8,w_9/w_8)$, for certain polynomial functions $w_i\in\f_3[x,y,z]$.
\item $R(q)$ has exactly two short orbits on $\cRq$.
 One is non-tame of size $q^3+1$, consisting of all $\fq$-rational places. The other is tame of size $q^3(q-1)(q+1)(q+3q_0+1)$, consisting of all $\mathbb F_{q^6}\setminus\mathbb F_q$-rational places.
\end{itemize}

Skabelund \cite[Sec. 4]{Sk2016} introduced the $\f_{q^6}$-maximal curve $\tcRq$ defined over $\fq$ with affine equations
$$
\tilde{\mathcal R}_q:\left\{
\begin{array}{ll}
t^m = x^q - x\\
z^q - z=x^{2q_0}\left( x^q - x \right)\\
y^q - y=x^{q_0}\left( x^q - x \right)
\end{array}
\right.,
$$
where $m=q-3q_0+1$. The curve $\tcRq$ is a degree-$m$ Galois cover of $\cRq$, with cyclic Galois group generated by $\gamma_\lambda:(x,y,z,w)\mapsto(x,y,z,\lambda w)$, where $\lambda\in\f_{q^4}$ is a primitive $m$-th root of unity.
Then, by the Riemann-Hurwitz formula, the genus of $\tcRq$ is $\frac{1}{2}\left(q^4-2q^3+q\right)$.
The set of $\fq$-rational places of $\tcRq$ has size $q^3+1$. It consists of the places centered at the affine points of $\tcRq$ lying on the plane $t=0$, together with the $\fq$-rational infinite place. Also, $\tcRq$ has no $\mathbb F_{q^2}$- or $\mathbb F_{q^3}$-rational places which are not $\mathbb F_q$-rational.

The automorphism group ${\rm Aut}(\tcRq)$ of $\tcRq$ has the following subgroups:

\begin{itemize}
\item A cyclic group $C_m$ generated by the automorphism $\gamma_\lambda:(x,y,t)\mapsto(x,y,\lambda t)$, where $\lambda\in\f_{q^4}$ is a primitive $m$-th root of unity;
\item A group $LR(q)$ lifted by $R(q)$ and generated by the automorphisms $\tilde{\psi}_{a,b,c,d}$ ($a,b,c,d\in\fq$, $a\ne0$) together with the involution $\tilde{\phi}$. 
Here, $\tilde{\psi}_{a,b,c,d}(x,y,z):=\psi_{a,b,c,d}(x,y,z)$ and $\tilde{\psi}_{a,b,c,d}(t):=\delta t$, where $\delta^m=a$.
Similarly, $\tilde{\phi}(x,y,z):=\phi(x,y,z)$, and $\tilde{\phi}(t):=t/w_8$; see \cite[Section 4]{Sk2016}.
\end{itemize}

\section{The automorphism group of $\tcSq$}\label{Sec:Suzuki}
We use the notations of Section \ref{Sec:PreliminaryResultsSuzuki}. 

\begin{lemma}\label{LiftedSuzuki}
The lifted group $LS(q)$ contains a subgroup $\tSq$ isomorphic to the Suzuki group $\Sq$.
\end{lemma}

\begin{proof}
Let $\Delta:=\{(\tilde{\psi}_{a,b,c})^m\,\mid\,a,b,c\in\fq,a\ne0\}\leq LS(q)$.
By direct checking, the map $\psi_{a,b,c}\mapsto(\tilde{\psi}_{a,b,c})^m$ is an isomorphism between $S(q)_{P_\infty}$ and $\Delta$.
Moreover, the action of $\Delta$ on the set $\cO$ of $\mathbb{F}_q$-rational places of $\tcSq$ is equivalent to the action of $S(q)_{P_\infty}$ on the non-tame short orbit of $S(q)$.
Let $\tSq$ be the subgroup of $LS(q)$ generated by $\Delta$ and $\tilde{\phi}$.
The action of $S(q)_\infty$ and $\phi$ on the non-tame short orbit of $S(q)$ is equivalent to the action of $\Delta$ and $\tilde\phi$ on $\cO$, respectively; hence, $\Delta$ coincides with the stabilizer in $\tSq$ of a point in $\cO$. 
This implies that $\tSq$ acts $2$-transitively on $\cO$
 and the stabilizer in $\tSq$ of two distinct places of $\cO$ is cyclic.
Since $|\cO|$ is not a power of $2$, we have by \cite[Theorem 1.7.6]{Bighearted} that $\tSq$ has no regular normal subgroups. Therefore we apply \cite[Theorem 1.1]{KOS} to conclude that $\tSq\cong\Sq$. 
\end{proof}

\begin{lemma}\label{NormalizerSuzuki}
The normalizer of $C_m$ in ${\rm Aut}(\tcSq)$ is the direct product $\tSq\times C_m$.
\end{lemma}

\begin{proof}
It is easily checked that $\gamma_\lambda$ commutes with $(\tilde{\psi}_{a,b,c})^m$ and with $\tilde{\phi}$ on the rational functions $x$, $y$ and $w$. 
Therefore, $\tSq\times C_m$ is a subgroup of the normalizer $N$ of $C_m$ in ${\rm Aut}(\tcSq)$; in particular, $N/C_m$ has a subgroup isomorphic to $\Sq$.
Also, the quotient curve $\tcSq/C_m$ is birationally equivalent to $\cSq$. Then $N/C_m$ is isomorphic to a subgroup of $\Sq$.
This implies that $N/C_m\cong\Sq$, whence the thesis.
\end{proof}

\begin{corollary}\label{CorollarySuzuki}
The group $LS(q)$ coincides with the normalizer of $C_m$ in ${\rm Aut}(\tcSq)$.
\end{corollary}

\begin{proof}
The group $C_m$ is contained in $LS(q)$ as it is generated by $\tilde{\psi}_{1,0,0}$. Also, $C_m$ commutes with every element of $LS(q)$. Hence, the claim follows by Lemma \ref{NormalizerSuzuki}.
\end{proof}

Being $\tcSq$ an $\f_{q^4}$-maximal curve, we can apply the results in \cite{GK2009} on zero $2$-rank curves. By direct computations $|{\rm Aut}(\tcSq)|\geq|LS(q)|\geq72(g(\tcSq)-1)$, then by \cite[Theorem 5.1]{GK2009} 
we conclude that ${\rm Aut}(\tcSq)$ is non-solvable. Applying \cite[Theorem 6.1]{GK2009}, the commutator ${\rm Aut}(\tcSq)^\prime$ of ${\rm Aut}(\tcSq)$ is one of the following groups:
$$ {\rm PSL}(2,n),\;{\rm PSU}(3,n),\;{\rm SU}(3,n),\;S(n)\quad{\rm with}\quad n=2^r\geq4. $$
Also, ${\rm Aut}(\tcSq)^\prime$ contains $G^\prime = \tilde{S}(q)$.

\begin{lemma}\label{CommutatorSuzuki}
${\rm Aut}(\tcSq)^\prime = \tilde{S}(q)$.
\end{lemma}

\begin{proof}
Using that $\tilde{S}(q)\leq{\rm Aut}(\tcSq)^\prime$, we discard the cases ${\rm PSL}(2,2^r)$, ${\rm PSU}(3,2^r)$, ${\rm SU}(3,2^r)$.
\begin{itemize}
\item[i)] $\tSq$ has elements of order $4$, while ${\rm PSL}(2,2^r)$ has not (see \cite[Hauptsatz 8.27]{Huppert}). Hence, ${\rm Aut}(\tcRq)^\prime \ne {\rm PSL}(2,2^r)$.
\item[ii)] $\tSq$ has subgroups which are a semidirect product $\Sigma\rtimes C_4$, where $\Sigma$ is generated by a tame element of order $q+2q_0+1$; see \cite{GKT2006}. On the contrary, in ${\rm PSU}(3,2^r)$ no non-tame element $\sigma$ of order $4$ can normalize a tame element $\tau$; otherwise, $\sigma$ acts on the fixed points of $\tau$ and in particular $\sigma$ fixes a point $P$ and a line $\ell$ not through $P$, which is impossible (see \cite[Lemma 2.2]{MZ2016}). Hence, $\tilde{G}^\prime \ne {\rm PSU}(3,2^r)$.
\item[iii)] If $\tilde{G}'={\rm SU}(3,2^r)$, then ${\rm SU}(3,2^r)$ has a subgroup of type $\Sigma\rtimes C_4$, where $\Sigma$ is cyclic of order $q+2q_0+1$; this implies that ${\rm PSU}(3,2^r)$ has a subgroup of type $\bar\Sigma\rtimes C_4$, where $\bar\Sigma$ is cyclic of order $(q+2q_0+1)/\gcd(3,2^r+1)$. This is impossible as shown at point $ii)$. Hence, $\tilde{G}^\prime \ne {\rm SU}(3,2^r)$.
\end{itemize} 
Therefore, ${\rm Aut}(\tcSq)^\prime = \tilde{S}(2^r)$.
If $2^r>q$, then $2^r\geq q^3$ and by direct computation $|{\rm Aut}(\tcSq)^\prime|>8 g(\tcSq)^3$; this is impossible by \cite[Theorem 11.116]{HKT}. Hence, ${\rm Aut}(\tcSq)^\prime = \tSq$.
\end{proof}

{\bf Theorem \ref{ThSuzuki}.}
{\it The automorphism group of $\tilde{\mathcal S}_q$ is a direct product $\tilde{S}(q)\times C_m$, where $\tilde{S}(q)$ is isomorphic to the Suzuki group $\Sq$ and $C_m$ is a cyclic group of order $m=q-2q_0+1$.}
\begin{proof}
By Lemma \ref{CommutatorSuzuki} and \cite[Theorem 6.2]{GK2009}, we have that ${\rm Aut}(\tcSq) \cong \tSq\times C$, where $C$ is a cyclic group of odd order. More specifically, $C$ is the subgroup of ${\rm Aut}(\tcSq)$ fixing pointwise the set $\mathcal{O}$ of $\fq$-rational places of $\tcSq$;
in particular, $C_m \subseteq C$. Then $C=C_m$ by Corollary \ref{CorollarySuzuki}.
\end{proof}

\begin{remark}
Theorem {\rm \ref{ThSuzuki}} shows that the automorphism group of $\tcSq$ is exactly the lifting $LS(q)$  obtained as a cyclic extension of the automorphism group of the Suzuki curve $\cSq$. 
\end{remark}

\section{The automorphism group of $\tcRq$}\label{Sec:Ree}

We use the notations of Section \ref{Sec:PreliminaryResultsRee} and start by recalling some results on large automorphism groups of curves that will be used in the proof of Theorem \ref{ThRee}.


\begin{theorem}{\rm (\cite[Theorems 11.56 and 11.116]{HKT})}\label{nothur}
Let $\mathcal X$ be an irreducible curve of genus $g \geq 2$ such that $|Aut(\mathcal X)| > 84(g-1).$ Then $Aut(\mathcal X)$ has at most three short orbits as follows:
\begin{enumerate}[i)]
\item exactly three short orbits, two tame and one non-tame and $|{\rm Aut}(\mathcal X)|\leq 24 g^2$; 
\item exactly two short orbits, both non-tame and $|{\rm Aut}(\mathcal X)|\leq 16 g^2$;
\item only one short orbit which is non-tame and $|{\rm Aut}(\mathcal X)|\leq g(2g-2)(4g+2) $ (see {\rm \cite[page 515]{HKT}});
\item exactly two short orbits, one tame and one non-tame. In this case $|{\rm Aut}(\mathcal X)|< 8g^3$, with the following exceptions (see {\rm \cite[Theorem 11.126]{HKT}}):
\begin{itemize}
\item $p=2$ and $\mathcal X$ is isomorphic to the hyperelliptic curve $Y^2+Y=X^{2^k+1}$ with genus $2^{k-1}$ ;
\item $p>2$ and $\mathcal X$ is isomorphic to the Roquette curve $Y^2=X^q-X$ with genus $(q-1)/2$ ;
\item $p\geq 2$ and $\mathcal X$ is isomorphic to the Hermitian curve $Y^{q+1}=X^q+X$ with genus $(q^2-q)/2$ ;
\item $p=2$ and $\mathcal X$ is isomorphic to the Suzuki curve $Y^q+Y=X^{q_0}(X^q+X)$ with genus $q_0(q-1)$ .
\end{itemize}
\end{enumerate}
\end{theorem}

\begin{remark}\label{NotException}
If $\mathcal X$ is the curve $\tcRq$ and Case {\it iv)} of Theorem \ref{nothur} occurs, then $|{\rm Aut}(\tcRq)|<8g^3$.
In fact, since $p=3$ and $g=\frac{3}{2}q_0(q-1)(q+q_0+1)$, $\tcRq$ cannot satisfy any of the four exceptions.
\end{remark}

Theorem \ref{Henn} provides a deeper analysis of Case {\it iv)} in Theorem \ref{nothur}; the bounds for the order of automorphism groups are taken from the proof of \cite[Theorem $11.116$]{HKT}.

\begin{theorem}\label{Henn}{\rm (\cite[Theorem 11.116 and page 516]{HKT})}
Suppose that Case {\it iv)} in Theorem {\rm \ref{nothur}} occurs.
Then one of the following cases holds:
\begin{enumerate}
\item 
$|{\rm Aut}(\mathcal X)|\leq 8g(g-1)(g+1)$ (see {\rm \cite[Eq. (11.169)]{HKT}}).
\item 
${\rm Aut}(\mathcal X)$ contains $p$-elements stabilizing two distinct places.
\item 
$|{\rm Aut}(\mathcal X)|\leq 8(g+1)(g-1)$ (see {\rm \cite[pages 524-525]{HKT}}).
\item 
The non-tame short orbit of ${\rm Aut}(\mathcal X)$ has length $p^k+1$ for some $k$ (see {\rm \cite[Lemma 11.123]{HKT}}).
\end{enumerate}
\end{theorem}
Cases {\it 1.}, {\it 3.}, and {\it 4.} in Theorem \ref{Henn} correspond to Case (iv1), (iv4), and (iv5) in \cite[page 516]{HKT}, respectively; Case {\it 2.} in Theorem \ref{Henn} corresponds to Cases (iv2) and (iv3) in \cite[page 516]{HKT}.

In analogy with Section \ref{Sec:Suzuki}, the following results hold.
The proofs of Lemma \ref{LiftedRee}, Lemma \ref{NormalizerRee}, and Corollary \ref{CorollaryRee} are analogous to the proofs of Lemma \ref{LiftedSuzuki}, Lemma \ref{NormalizerSuzuki}, and Corollary \ref{CorollarySuzuki} in Section \ref{Sec:Suzuki}.

\begin{lemma}\label{LiftedRee}
The lifted group $LR(q)$ contains a subgroup $\tRq$ isomorphic to the Ree group ${\rm Aut}(\cRq)$.
\end{lemma}

\begin{lemma}\label{NormalizerRee}
The normalizer of $C_m$ in ${\rm Aut}(\tcRq)$ is the direct product $\tRq\times C_m$.
\end{lemma}

\begin{corollary}\label{CorollaryRee}
The group $LR(q)$ coincides with the normalizer of $C_m$ in ${\rm Aut}(\tcRq)$.
\end{corollary}


\begin{proposition}\label{OrbitsOfG}
The group $LR(q)$ has exactly two short orbits $\cO_T$ and $\cO_{NT}$ in its action on $\tcRq$. The orbit $\cO_T$ is tame of size $(q^3+1)q^3(q-1)$, consisting of the $\mathbb F_{q^6}\setminus\mathbb F_q$-rational places; the orbit $\cO_{NT}$ is non-tame, consisting of the $q^3+1$  $\fq$-rational places of $\tcRq$.
\end{proposition}

\begin{proof}
The set $\cO$ of the $\mathbb{F}_q$-rational places of $\tcRq$ is the non-tame short orbit $\cO_{NT}$ under $LR(q)$, since $C_m$ acts trivially on $\cO$.

Now, let $\cO_{T}\subseteq\tcRq$ be the set of $\mathbb F_{q^6}\setminus\mathbb F_q$-rational places; we prove that $\cO_T$ is a tame short orbit under $LR(q)$.
Let $P\in\cO_{T}$. Since $C_m$ is defined over $\f_{q^6}$, the place $Q \in\cRq$ lying under $P$ has degree $1$, $2$, $3$, or $6$.
The places of $\cRq$ of degree $1$ lie under a place in $\cO_{NT}$, and $\cRq$ has no places of degree $2$ or $3$; therefore, $Q$ has degree $6$.
By the fundamental equality \cite[Theorem 3.1.11]{Sti}, we conclude that there are exactly $m$ $\mathbb F_{q^6}\setminus\mathbb F_q$-rational places of $\tcRq$ lying over an $\mathbb F_{q^6}\setminus\mathbb F_q$-rational place of $\cRq$.
By the $\f_{q^6}$-maximality of $\tcRq$, we have that $|\cO_T|=m q^3(q-1)(q+1)(q+3q_0+1)$; hence, $\cO_T$ coincides with the set of places of $\tcRq$ lying over an $\mathbb F_{q^6}\setminus\mathbb F_q$-rational place of $\cRq$.

To show that $LR(q)$ is transitive on $\cO_T$, let $P_1,P_2\in\cO_T$ with $P_1\ne P_2$. If $P_1$ and $P_2$ are in the same $C_m$-orbit, the claim is proved. Otherwise, let $Q_1$ and $Q_2$ be the distinct places of $\cRq$ lying under $P_1$ and $P_2$, respectively.
Since $Q_1$ and $Q_2$ are in the tame short orbit of $\cRq$ under $R(q)$, there exists $\sigma\in R(q)$ such that $\sigma(Q_1)=Q_2$. Let $\tilde\sigma$ be the induced automorphism of $\tcRq$, and let $P_3:=\tilde\sigma(P_1)$. Then $P_3$ is in the $C_m$-orbit of $P_2$, because $\cSq$ is $\tcSq/C_m$. Let $\tau\in C_m$ with $\tau(P_3)=P_2$; then $\tau\tilde\sigma(P_1)=P_2$.

Since $R(q)$ acts semiregularly on the non-$\mathbb F_{q^6}$-rational places of $\cRq$, $LR(q)$ acts semiregularly on $\tcRq\setminus\left(\cO_T\cup\cO_{NT}\right)$, and the thesis is proved.
\end{proof}


Let $\tilde{\cO}_{NT}$ be the non-tame short orbit of $\tcRq$ under ${\rm Aut}(\tcRq)$ containing $\cO_{NT}$.

\begin{lemma} 
The orbit $\tilde\cO_{NT}$ coincides with $\cO_{NT}$.
\end{lemma}
\begin{proof}
Suppose by contradiction that $\cO_{NT}\ne{\tilde \cO}_{NT}$.

Firstly, suppose that $\tilde\cO_{NT}\setminus\cO_{NT}$ contains a long orbit under $LR(q)$.
Then, for any $\fq$-rational place $P\in\tcRq$ we have
\begin{align*}
|{\rm Aut}(\tcRq)| & =|\tilde{\cO}_{NT}|\cdot|{\rm Aut}(\tcRq)_P| \geq |LR(q)|\cdot|LR(q)_P|\\
& \geq (q^3+1)q^3(q-1)m\cdot q^3(q-1)m > 8g^3,  
\end{align*}
where $g$ is the genus of $\tcRq$ and $LR(q)_P$ denotes the stabilizer of the place $P$ in the group $LR(q)$.
Since $|{\rm Aut}(\tcRq)|>84(g-1)$, this is impossible by Theorem \ref{nothur} and Remark \ref{NotException}.
Then $\tilde{\cO}_{NT}\setminus\cO_{NT}$ contains a short orbit under $LR(q)$ and $\tilde{\cO}_{NT}=\cO_{NT}\cup\cO_T$ by Proposition \ref{OrbitsOfG}.

If ${\rm Aut}(\tcRq)_P\ne LR(q)_P$, then $|{\rm Aut}(\tcRq)_P| \geq 2|LR(q)_P|$, and hence
$$ |{\rm Aut}(\tcRq)|=|\tilde{\cO}_{NT}|\cdot|{\rm Aut}(\tcRq)_P| \geq |\tilde{\cO}_{NT}|\cdot2|LR(q)_P| \geq |\cO_{T}| \cdot 2q^3(q-1)m > 8g^3,$$
which is impossible by Theorem \ref{nothur}. Therefore, ${\rm Aut}(\tcRq)_P = LR(q)_P$. This implies
$$
|{\rm Aut}(\tcRq)|=|\tilde{\cO}_{NT}|\cdot|LR(q)_P| = (q^3+1)q^3(q-1)(q-3q_0+1)(q^4-q^3+1).
$$
Note that the order of $|{\rm Aut}(\tcRq)|$ is very close to $8g^3$.

Since $|{\rm Aut}(\tcRq)|>g(2g-2)(4g+2)$, Cases {\it i)}, {\it ii)}, and {\it iii)} in Theorem \ref{nothur} cannot occur, hence Case {\it iv)} holds and one of Cases {\it 1.} - {\it 4.} in Theorem \ref{Henn} occurs.
\begin{itemize}
\item Since $|{\rm Aut}(\tcRq)|>8g(g-1)(g+1)$, Cases {\it 1.} and {\it 3.} cannot occur.
\item Case {\it 2.} cannot occur; in fact, $\tcRq$ has zero $p$-rank, and hence any $p$-element in ${\rm Aut}(\tcRq)$ has exactly one fixed place.
\item Case {\it 4.} cannot occur, since $|\tilde{\cO}_{NT}|=(q^3+1)(q^4-q^3+1)\ne 3^k+1$ for any $k$.
\end{itemize}
\end{proof}

Finally we prove Theorem \ref{ThRee}.

{\bf Theorem \ref{ThRee}.}
{\it The automorphism group of $\tilde{\mathcal R}_q$ is isomorphic to $\tilde{R}(q)\times C_m$, where $\tilde{R}(q)\cong{\rm Aut}(\mathcal R_q)$ is the Ree group and $C_m$ is a cyclic group of order $m$.}

\begin{proof}
Let $\alpha\in{\rm Aut}(\tcRq)$, and define  
$T:=\{\sigma \in {\rm Aut}(\tcRq) \,\mid\, \sigma(P)=P\;\;\textrm{for all}\;\;P\in\tilde{\cO}_{NT} \}$ and $C_m^\prime:=\alpha C_m \alpha^{-1}$.
Clearly, $T$ contains $C_m$ and $C_m^\prime$.
By \cite[Lemma 11.129]{HKT}, $T$ is a tame subgroup of ${\rm Aut}(\tcRq)$, which implies that $T$ is cyclic (see \cite[Lemma 11.44]{HKT}).
Therefore $C_m^\prime = C_m$, that is, $C_m$ is normal in ${\rm Aut}(\tcRq)$. Corollary \ref{CorollaryRee} yields the thesis.
\end{proof}

\begin{remark}
Theorem {\rm \ref{ThRee}} shows that the automorphism group of $\tcRq$ is exactly the lifting described above and obtained as a cyclic extension of the automorphism group of the cyclic subcover $\cRq$. 
\end{remark}

\section{Non-existence of certain Galois coverings}
\label{Sec:NonGaloisCovering}

{\bf Theorem \ref{CoveredSuzuki}.}
{\it For any $q$, the curve $\tilde{\mathcal S}_q$ is not Galois covered by the Hermitian curve $\mathcal{H}_{q^2}$.}

\begin{proof}
Suppose by contradiction that $\tilde{\mathcal S}_q$ is a Galois subcover of $\mathcal{H}_{q^2}$, that is $\tcSq\cong\mathcal{H}_{q^2}/G$ with $G\leq{\rm PGU}(3,q^2)$.
The different divisor has degree
\begin{equation}\label{DifferentSuzuki}
\Delta =  (2g(\mathcal{H}_{q^2})-2) - |G|(2g(\tcSq)-2) = q^4-q^2-2 - |G|(q^3-2q^2+q-2)\,.
\end{equation}
By the Riemann-Hurwitz formula,
$$ \frac{q^6+1}{q^5-q^4+q^3+1} = \frac{|\mathcal{H}_{q^2}(\mathbb{F}_{q^4})|}{|\tilde{\mathcal S}_q(\mathbb{F}_{q^4})|}\leq |G| \leq \frac{2g(\mathcal{H}_{q^2})-2}{2g(\tilde{\mathcal S}_q)-2}= \frac{q^4-q^2-2}{q^3-2q^2+q-2} \,, $$
hence $q+1\leq|G|\leq q+2$.

Suppose $|G|=q+1$.
By \cite[Theorem 2.7]{MZ2016}, we have $\Delta=q\cdot2$. This contradicts Equation \eqref{DifferentSuzuki}, which reads $\Delta=q^3+q$.

For $q>8$, $|G|\ne q+2$ because $|G|$ divides $|{\rm PGU}(3,q^2)|=(q^6+1)q^6(q^4-1)$.
For $q=8$, suppose $|G|=q+2=10$.
By \cite[Lemma 2.2]{MZ2016}, the generator $\alpha$ of the unique Sylow $5$-subgroup $C_5$ is either of type (A) or (B1); hence, $\alpha$ fixes a point $P$ and a line $\ell$ not through $P$.
Since $C_5$ is normal in $G$, the generator $\beta$ of any Sylow $2$-subgroup $C_2$ of $G$ fixes $P$ and $\ell$.
Therefore, $\beta$ cannot be of type (D); thus, $\beta$ is of type (C). This implies that $\alpha$ is not of type (B1), and hence $\alpha$ is of type (A).
Then $\Delta\geq4\cdot65$ by \cite[Theorem 2.7]{MZ2016}. This contradicts Equation \eqref{DifferentSuzuki}.
\end{proof}

Now consider the curve $\tcRq$.
Suppose that $\tcRq\cong\mathcal{H}_{q^3}/G$ for some $G\leq{\rm PGU}(3,q^3)$.
The different divisor has degree
\begin{equation}\label{Different2}
\Delta =  (2g(\mathcal{H}_{q^3})-2) - |G|(2g(\tcRq)-2) = q^6-q^3-2 - |G|(q^4-2q^3+q-2)\,.
\end{equation}
By the Riemann-Hurwitz formula,
$$ \frac{q^9+1}{q^7-q^6+q^4+1} = \frac{|\mathcal{H}_{q^3}(\mathbb{F}_{q^6})|}{|\tilde{\mathcal R}_q(\mathbb{F}_{q^6})|}\leq |G| \leq \frac{2g(\mathcal{H}_{q^3})-2}{2g(\tilde{\mathcal R}_q)-2}= \frac{q^6-q^3-2}{q^4-2q^3+q-2} \,, $$
hence
$$q^2+q+1\leq|G|\leq q^2+2q+4.$$

\begin{lemma}\label{PossibleValuesRee}
If $\tcRq\cong\mathcal{H}_{q^3}/G$, then $$|G|\,\mid\,|{\rm PGU}(3,q^3)|\,,\;\; q^2+q+1\leq|G|\leq q^2+2q+4\,,\;\;|G|\notin\{q^2+q+1,q^2+2q+1\}.$$
\end{lemma}

\begin{proof}
{\bf Case $|G|=q^2+q+1$.} Since $|G|$ divides $q^3-1$ and is coprime with $q^3+1$, we have by \cite[Theorem 2.7]{MZ2016} that $\Delta = 2(q^2+q)$. This contradicts Equation \eqref{Different2}.

{\bf Case $|G|=q^2+2q+1=(q+1)^2$.}
By \cite[Theorem A.10]{HKT}, ${\rm PGU}(3,q^3)$ contains only two conjugacy classes of maximal subgroups whose order is divided by $|G|$:
\begin{itemize}
\item The stabilizer $M_1$ of a self-conjugate triangle $T$, of order $|M_1|=6(q^3+1)^2$.
\item The stabilizer $M_2$ of a non-tangent line $\ell$, of order $|M_2|=q^3(q^6-1)(q+1)$.

The center $Z$ of $M_2$ has order $q^3+1$ and is a cyclic group of homologies acting trivially on $\ell$. The group $M_2/Z$ acts faithfully on $\ell$ as a linear group, hence it is isomorphic to a subgroup of ${\rm PGL}(2,q^6)$. $M_2/Z$ acts on the $q^3+1$ points of $\ell\cap\mathcal{H}_{q^3}$; by the structure of $M_2$, we have that $M_2\cong{\rm PGL}(2,q^3)$, and the action of $M_2$ on $\ell\cap\mathcal{H}_{q^3}$ is equivalent to the action of ${\rm PGL}(2,q^3)$ in its natural $2$-transitive permutation representation.
\end{itemize}
Suppose that $G\subseteq M_2$. The group $G/(Z\cap G)$ acts faithfully on $\ell$ and is isomorphic to a subgroup of ${\rm PGL}(2,q^3)$. Since $|G\cap Z|$ is a divisor of $q+1$, we have that $|G/(G\cap Z)|=(q+1)d,$ where $d$ divides $q+1$ and we conclude that $|G/(G\cap Z)|$ is equal to $q+1$ or $2(q+1)$ since $|{\rm PGL}(2,q^3)|=q(q^2-1)$.
From \cite[Theorem A.8]{HKT} the group $G/(G\cap Z)$ is cyclic of order $q+1$ or the dihedral group of order $2(q+1)$.

In particular, one of the following cases occurs:
\begin{itemize}
\item $G/(Z\cap G)$ is a cyclic Singer group with two fixed points $P_1,P_2$ on $\ell\setminus\mathcal{H}_{q^3}$. The pole $P_3$ of $\ell$ is also fixed by $G$; hence, $G$ fixes a self-conjugate triangle $T$.
\item $G/(Z\cap G)$ is a dihedral group normalizing a cyclic Singer group $S$ of index $2$, such that $S$ fixes two points $P_1,P_2$ on $\ell\setminus\mathcal{H}_{q^3}$; hence, $G$ switches $P_1$ and $P_2$ and fixes the pole $P_3$ of $P_1P_2$. Then $G$ fixes a self-conjugate triangle $T$.
\end{itemize}

Then $G\subseteq M_1$.

Up to conjugation, $T$ is the fundamental triangle, so that
$$ M_1=\{{\rm diag}(\lambda,\mu,1)\,\mid\, \lambda^{q^3+1}=\mu^{q^3+1}=1\} \rtimes Sym(3), $$
where $Sym(3)$ is the group of $3\times3$ permutation matrices.
The only subgroup of order $(q+1)^2$ in $M_1$ is
$$ G=\{{\rm diag}(\lambda,\mu,1)\,\mid\, \lambda^{q+1}=\mu^{q+1}=1\} \cong C_{q+1}\times C_{q+1}. $$

With the notations of \cite[Lemma 2.2]{MZ2016}, $G$ contains exactly $3q$ elements of type (A) and $q^2-q$ elements of type (B1). Then by \cite[Theorem 2.7]{MZ2016} we have $\Delta=3q(q^3+1)$. The same value for $\Delta$ is obtained by Equation \eqref{Different2}, that is, the curves $\mathcal{H}_{q^3}/G$ and $\tcRq$ actually have the same genus.

The group $G$ is normal in $M_1$, thus $M_1/G$ is an automorphism group of $\mathcal{H}_{q^3}/G$ of order $|M_1/G|=6(q^2-q+1)^2$. Since $|M_1/G|$ is not a divisor of $|{\rm Aut}(\tcRq)|=(q^3+1)q^3(q-1)(q-3q_0+1)$, we have $\mathcal{H}_{q^3}/G\not\cong\tcRq$.
\end{proof}

By the proof of Lemma \ref{PossibleValuesRee} we get the following remark.

\begin{remark}\label{NotIsomorphic}
For any odd power $q\geq27$ of $3$, let $G\leq{\rm PGU}(3,q^3)$ with $|G|=(q+1)^2$ ($G$ is unique up to conjugation).
Then the curves $\mathcal{H}_{q^3}/G$ and $\tcRq$ have the same genus but are not isomorphic, since they have different automorphism groups.
\end{remark}

\begin{theorem}
For any $q$, the curve $\tcRq$ is not Galois covered by the Hermitian curve $\mathcal{H}_{q^3}$.
\end{theorem}

\begin{proof}
Suppose by contradiction that $\tcRq\cong\mathcal{H}_{q^3}/G$.
By Lemma \ref{PossibleValuesRee}, the order of $G$ satisfies $q^2+q+2\leq |G| \leq q^2+2q+4$ and $|G|\ne q^2+2q+1$.
By Equation \eqref{Different2} $\Delta$ is a multiple of $q^3+1$.
This fact, together with \cite[Theorem 2.7]{MZ2016} and $3|G|<q^3+1$, implies that $i(\sigma)\in\{0,q^3+1\}$ for any nontrivial $\sigma\in G$, that is, $\sigma$ is of type (A) or (B1) and the order of $\sigma$ divides $q^3+1$.

By \cite[Theorem A.10]{HKT}, $G$ is contained in the stabilizer $N\leq{\rm PGU}(3,q^3)$ of a self-conjugate triangle, hence $G$ acts on three non-collinear points $\{P_1,P_2,P_3\}$ of ${\rm PG}(2,q^6)\setminus\mathcal{H}_{q^3}$.
In fact, because of its order, $G$ can be only be contained in the following maximal subgroups of ${\rm PGU}(3,q^3)$ other than $N$:
\begin{enumerate}
\item The stabilizer of one point $P\in\mathcal{H}_{q^3}(\mathbb{F}_{q^6})$. In this case, $G$ cannot contain elements of type (B1); hence $\Delta=(q^3+1)(|G|-1)$, exceeding the value in \eqref{Different2}.
\item The stabilizer of a point $P\in {\rm PG}(2,q^6)\setminus\mathcal{H}_{q^3}$ and its non-tangent polar line $\ell$. In this case, either $G$ acts trivially on $\ell$, or $G$ fixes two points $Q,R\in\ell\setminus\mathcal{H}_{q^3}$ by \cite[Hauptsatz 8.27]{Huppert}. In the former case, $\Delta$ exceeds the value in \eqref{Different2}. In the latter case, $G$ fixes the self-conjugate triangle $\{P,Q,R\}$.
\item A group isomorphic to ${\rm PGL}(2,q^3)$. In this case, by \cite[Hauptsatz 8.27]{Huppert}, $G$ contains a cyclic subgroup $G^\prime$ of index $1$ or $2$. In any case, $G$ stabilizes the self-conjugate triangle whose vertices are fixed by $G^\prime$.
\item A group isomorphic to ${\rm PGU}(3,q)$. This is impossible since $G$ cannot divide the order of any maximal subgroup of ${\rm PGU}(3,q)$.
\end{enumerate}
Note that, since $G$ is not divisible by $3$, $G$ fixes at least one point in $\{P_1,P_2,P_3\}$, say $P_1$, and acts on $\{P_2,P_3\}$.
Let $Z\leq N$ be the subgroup of homologies (that is, elements of type (A) ) with center $P_1$ and axis $P_2P_3$; $Z$ is cyclic of order $q^3+1$ and is the center of $N$.
By direct computation, there exists a divisor $d>2$ of $q+1$ which is coprime with $|G|$.
Then the normalizer of $G$ in ${\rm PGU}(3,q^3)$ contains the subgroup $D$ of $Z$ of order $d$.
Therefore $D$ induces a cyclic automorphism group $\bar D$ of $\mathcal H_{q^3}/G\cong \tcRq$ of order $d$ which fixes at least one point of $\tcRq$.
The automorphism group of $\tcRq$ has exactly two short orbits $\cO_{T}$ and $\cO_{NT}$ of size $(q^3+1)q^3(q-1)$ and $q^3+1$; see Proposition \ref{OrbitsOfG}.
Then $d$ divides $|{\rm Aut}(\tcRq)|/|\cO_{T}|$ or $|{\rm Aut}(\tcRq)|/|\cO_{NT}|$. By direct checking, this is impossible.
\end{proof}

\section{Galois subcovers of $\tcSq$}\label{Sec:QuotientsSuzuki}

Under the same notation as in Sections \ref{Sec:PreliminaryResultsSuzuki} and \ref{Sec:Suzuki}, we use the properties of $\aut(\tcSq) = \tSq\times C_m \cong \Sq \times C_m$ to obtain the genera of many Galois subcovers of $\tcSq$.

Let $\varphi$ be the rational morphism $\varphi : \tcSq \to \cSq$, $ \varphi(x,y,t)=(x,y)$, and $ \pi$ be the natural projection $\aut(\tcSq)\to\aut(\tcSq)/C_m\cong \Sq$. Clearly, $\varphi\circ h=\pi(h)\circ\varphi$ holds for any $h \in\aut(\tcSq)$. For $L$ a subgroup of $\aut(\tcSq)$ let $\bar L =\pi(L)= L/(L \cap C_m)$. Denote by $g_L$ the genus of the quotient curve $\tcSq/L$ and by $g_{\bar L}$ the genus of the quotient curve $\cSq/\bar L$.
By the Riemann-Hurwitz formula applied to the cover $\tcSq\to\tcSq/L$,
$$ (q^2+1)(q-2) = |L|(2g_L-2)+\Delta_L, $$
where $\Delta_L$ is the degree of the different divisor.

By the Hilbert different formula we have
$$ \Delta_L = \sum_{\sigma\in L,\sigma\ne id} i(\sigma), $$
where
$$ i(\sigma) = \sum_{P\in\tcSq\,:\,\sigma(P)=P} \left|\left\{i\in\{0,1,2,\ldots\}\,:\, \sigma\in L_{P}^{(i)}\right\}\right|. $$
Here, $L_{P}^{(i)}$ denotes the $i$-th ramification group of $L$ at $P$; see \cite[Definition 3.8.4]{Sti}.

If $L$ is tame (i.e. $2=p\nmid|L|$), then $L_P^{(i)}$ is trivial for $i\geq1$; hence,
$$ i(\sigma)=\left|\left\{P\in\tcSq\,:\,\sigma(P)=P\right\}\right|\quad\textrm{and}\quad \Delta_L=\sum_{P\in\tcSq}(|L_P|-1), $$
where $L_P=L_P^{(0)}$ is the stabilizer of $P$ in $L$.

\begin{lemma}
The set of $\mathbb{F}_{q^4}$-rational places of $\tcSq$ splits into two short orbits under the action of $\aut(\tcSq)$, one is non-tame of size $q^2+1$, and consists of the $\mathbb F_q$-rational places of $\tcSq$, the other is tame of size $q^2(q^2+1)(q-1)$, and consists of the $\mathbb F_{q^4}$-rational places of $\tcSq$.
\end{lemma}

\begin{proof}
Since $\aut(\tcSq)=LS(q)$, the proof is analogous to the proof of Proposition \ref{OrbitsOfG} for the curve $\tcRq$.
\end{proof}

Denote by $\cO_1$ and $\cO_2$ the non-tame and the tame short orbit of $\aut(\tcSq)$, respectively.
The non-tame and the tame short orbit of $\aut(\cSq)$ on $\cSq$ coincide respectively with the images $\bar\cO_1$ and $\bar\cO_2$ of $\cO_1$ and $\cO_2$ under $\varphi$. The places in $\bar\cO_1$ are totally ramified under $\cO_1$, while the places of $\bar\cO_2$ are totally unramified under $\cO_2$.

The proofs of the following results are omitted since they are analogous to the proofs of Propositions 3.1, 3.2 and Corollary 3.4 in \cite{FG}.

\begin{proposition}\label{generaldelta}
Let $L$ be a tame subgroup of $\aut(\tcSq)$ and $L_{C_m}:=L\cap C_m$.
Then
$$\Delta_L=(|L_{C_m}|-1)(q^2+1) +|L_{C_m}|n_1 + |L_{C_m}|n_2, $$
where
\begin{itemize}
\item $n_1$ counts the non-trivial relations $\bar h(\bar P)=\bar P$ with $\bar h\in\bar L$ when $\bar P$ varies in $\bar\cO_1$, namely
$$ n_1=\sum_{\bar h\in\bar L,\bar h\ne id} |\{\bar P\in\bar\cO_1 \mid \bar h(\bar P)=\bar P\}|; $$
\item  $n_2$ counts the non-trivial relations $\bar h(\bar P)=\bar P$ with $\bar h\in\bar L$ when $\bar P$ varies in $\bar \cO_2$, each counted with a multiplicity $l_{\bar h,\bar P}$ defined as the number of orbits of $\varphi^{-1}(\bar P)$ under the action of $L_{C_m}$ that are fixed by an element $h\in\pi^{-1}(\bar h)$. That is,
$$ n_2= \sum_{\bar h\in\bar L,\bar h\ne id} \sum_{\bar P\in\bar\cO_2,\bar h(\bar P)=\bar P} l_{\bar h,\bar P}. $$
\end{itemize}
\end{proposition}

\begin{proposition}
Let $L$ be a tame subgroup of $\aut(\tcSq)$. Assume that no non-trivial element in $\bar L$ fixes a place out of $\bar\cO_1$. Then
$$ g_L = g_{\bar L} + \frac{(q^2+1)(q-|L_{C_m}|-1)-2|L_{C_m}|(q_0 q - q_0 - 1)}{2|L|}. $$
\end{proposition}

\begin{proposition}
Let $L$ be a tame subgroup of $\aut(\tcSq)$ containing $C_m$. Then
$$ \Delta_L=(q-2q_0)(q^2+1) + (q-2q_0+1)\sum_{\bar h\in\bar L,\bar h\ne id}|\{\bar P\in \cSq(\mathbb{F}_{q^4})\mid \bar h(\bar P)=\bar P\}|. $$
\end{proposition}

\begin{proposition}\label{SameGenusSuzuki}
Let $L$ be a subgroup of $\aut(\tcSq)$ containing $C_m$. Then $g_L = g_{\bar L}$.
\end{proposition}

\begin{proof}
Since $\tcSq/C_m \cong \cSq$, the Galois group of the cover $\tcSq\to\cSq$ is $C_m$; hence the Galois group of the cover $\tcSq\to\cSq/\bar L$ is isomorphic to $\bar L\times C_m \cong L$. From the Galois correspondence, $\cSq/\bar L\cong\tcSq/L$ and the claim follows.
\end{proof}

We now compute the contributions $i(\sigma)$ to $\Delta_L$, starting with the description of the higher ramification groups.
Since $L_P^{(1)}$ is the Sylow $2$-subgroup of $L_P$, non-trival higher ramification groups at $P$ only exist for $P$ in the non-tame orbit $\cO_1$. Up to conjugation, we assume that $P$ is the infinite place $P_\infty$ of $\tcSq$.

For $a,b,c\in\mathbb F_q$, $a\ne0$, let $\tilde \psi_{a,b,c}\in\aut(\tcSq)_{P_\infty}$ be the automorphism described in Section \ref{Sec:PreliminaryResultsSuzuki}.
Let $\tilde\theta_{a,b,c}\in\aut(\tcSq)$ be the automorphism 
$\tilde\theta_{a,b,c}:(x,y,t)\mapsto(ax+b,a^{q_0+1}y+b^{q_0}x+c,at)$.
Then $\tSq_{P_\infty}=\{\tilde\theta_{a,b,c}\mid a,b,c\in\mathbb F_q,a\ne0\}$, as pointed out in the proof of Lemma \ref{LiftedSuzuki}.

By direct checking, $\tilde\theta_{1,b,c}$ has order $2$ or $4$ according to $b=0$ or $b\ne0$, respectively.

\begin{proposition}\label{SuzukiRamification}
Let $G=\aut(\tcSq)$, then 
$$G_{P_\infty}^{(i)}=
\begin{cases}
\tSq_{P_\infty} \times C_m, i=0, \\
\{\tilde\theta_{1,b,c}\mid b,c\in\mathbb F_q\}=\{\tilde\theta_{a,b,c}\mid \tilde\theta_{a,b,c}\;\textrm{has order }1,2,\textrm{or }4 \}, 1 \leq i \leq m,\\
\{\tilde\theta_{1,0,c}\mid c\in\mathbb F_q\}=\{\tilde\theta_{a,b,c}\mid \tilde\theta_{a,b,c}\;\textrm{has order }1\textrm{or }2\}, m+1 \leq i \leq m(2q_0+1),\\
 \{ Id \}, i>m(2q_0+1)=q^2+1-mq.
\end{cases}
$$
\end{proposition}

\begin{proof}
Since $G=\tSq \times C_m$ and $C_m$ fixes $P_\infty$, we have $G_{P_\infty}^{(0)} = G_{P_\infty} = \tSq_{P_\infty} \times C_m$. Let $w=xy^{2q_0}+z^{2q_0}$ with $z=x^{2q_0+1}+y^{2q_0}$ be functions in the function field of the curve $\tcSq$. The valuation of $w$ at the infinite place $\bar P_\infty$ of $\cSq$ is equal to $-(q+2q_0+1)$, see \cite[Eq. (3.7)]{GKT2006}; thus, $w$ has valuation $-(q+2q_0+1)m=-(q^2+1)$ at $P_\infty$.
The function $x$ has valuation $-q$ at $\bar P_\infty$ (see the proof of Theorem 3.1 in \cite{GKT2006}), and hence $-qm$ at $P_\infty$; since $t^m=x^q+x$, this implies that $t$ has valuation $-q^2$ at $P_\infty$.
Therefore, $t/w$ is a local parameter at $P_\infty$.
The first ramification group $G_{P_\infty}^{(1)}$ is the Sylow $2$-subgroup $\{\tilde\theta_{1,b,c}\mid b,c\in\mathbb F_q\}$ of $G_{P_\infty}$.
By direct computation,
$$ \tilde{\theta}_{1,b,c}\left(\frac{t}{w}\right)-\frac{t}{w} = -\frac{t}{w}\cdot\frac{\epsilon}{w+\epsilon}, $$
where $\epsilon = bz + c^{2q_0}x + bc^{2q_0} + c^2 + b^{2q_0+2}$.
From the proof of Theorem 3.1 in \cite{GKT2006} we have that the valuation of $z$ at $\bar P_\infty$ is $-(q+2q_0)$, and we conclude that the valuation of $\epsilon$ at $P_\infty$ is either $-mq$ or $-m(q+2q_0)$ according to $b=0$ or $b\ne0$, respectively.
Therefore, the valuation of $\tilde{\theta}_{1,b,c}\left(t/w\right)-t/w$ at $P_\infty$ is either $m(2q_0+1)+1$ or $m+1$ according to $b=0$ or $b\ne0$, respectively.
The claim follows.
\end{proof}

Up to conjugacy, the maximal subgroups of $S(q)$ are the following; see for instance \cite[Section 2]{GKT2006} and \cite[Theorem A.12]{HKT}.
\begin{itemize}
\item[(I)] The stabilizer of a $\mathbb F_q$-rational place of $\cSq$, of order $q^2(q-1)$.
\item[(II)] The normalizer $N_+$ of a cyclic Singer group $\Sigma_+$, with $|\Sigma_+|=q+2q_0+1$ and $|N_+|=4(q+2q_0+1)$; $\Sigma_+$ acts semiregularly on $\cSq$.
\item[(III)] The normalizer $N_-$ of a cyclic Singer group $\Sigma_-$, with $|\Sigma_-|=q-2q_0+1$ and $|N_-|=4(q-2q_0+1)$; $\Sigma_-$ fixes four $\mathbb F_{q^4}$-rational places of $\cSq$ and acts semiregularly elsewhere; $N_-$ is transitive on the fixed places of $\Sigma_-$.
\item[(IV)] The Suzuki group $S(q^\prime)$, where $q=2^{2s+1}$, $q^\prime=2^{2s^\prime+1}$, $s^\prime$ divides $s$, $2s^\prime+1$ divides $2s+1$, and $s/s^\prime$ is prime.
\end{itemize}

Moreover, the following subgroups form a partition of $S(q)$.
\begin{itemize}
\item All subgroups of order $q^2$.
\item All cyclic subgroups of order $q-1$.
\item All cyclic Singer subgroups of order $q+2q_0+1$.
\item All cyclic Singer subgroups of order $q-2q_0+1$.
\end{itemize}

\begin{theorem}\label{ContributionsSuzuki}
Let $\sigma\in\tSq\setminus\{id\}$ and $C_m=\langle\tau\rangle$.
Denote by $o(\sigma)$ the order of $\sigma$.
Then $i(\tau^k)=q^2+1$ for all $k=1,\ldots,m-1$ and one of the following cases occurs.
\begin{itemize}
\item $o(\sigma)=2$, $i(\sigma)=m(2q_0+1)+1$, and $i(\sigma\tau^k)=1$ for all $k=1,\ldots,m-1$;
\item $o(\sigma)=4$, $i(\sigma)=m+1$, and $i(\sigma\tau^k)=1$ for all $k=1,\ldots,m-1$;
\item $o(\sigma)\mid(q-1)$, $i(\sigma)=2$, and $i(\sigma\tau^k)=2$ for all $k=1,\ldots,m-1$;
\item $o(\sigma)\mid(q+2q_0+1)$, $i(\sigma)=0$, and $i(\sigma\tau^k)=0$ for all $k=1,\ldots,m-1$;
\item $o(\sigma)\mid(q-2q_0+1)$, $i(\sigma)=0$, $i(\sigma\tau^j)=4m$ for exactly one $j\in\{1,\ldots,m-1\}$, and $i(\sigma\tau^k)=0$ for all $k\in\{1,\ldots,m-1\}\setminus\{j\}$.
\end{itemize}
\end{theorem}

\begin{proof}
From the orbit-stabilizer theorem, the stabilizer in $\aut(\tcSq)$ of a place in $\cO_1$ has size $q^2(q-1)(q-2q_0+1)$, while the stabilizer of a place in $\cO_2$ has size $q-2q_0+1=m$.
\begin{itemize}
\item Since $\tau^k$ fixes $x$ and $y$ but does not fix $t$, the fixed places of $\tau^k$ are exactly the $q^2+1$ places in $\cO_1$. As $C_m$ is tame, $i(\tau^k)=q^2+1$.
\item Let $o(\sigma)\in\{2,4\}$. Since $\tcSq$ has zero $2$-rank, any $2$-element of $\aut(\tcSq)$ has exactly one fixed place, which is in $\cO_1$ as $2\nmid m$.
Up to conjugation, the place fixed by $\sigma$ is $P_\infty$, and the claim on $i(\sigma)$ follows from Proposition \ref{SuzukiRamification}.
The element $\sigma\tau$ fixes $P_\infty$ and no other place in $\cO_1$. As $o(\sigma\tau^k)\nmid m$, $P_\infty$ is the only fixed place of $\sigma\tau^k$. Moreover $\sigma\tau^k\notin\aut(\tcSq)_{P_\infty}^{(1)}$, since $\sigma\tau^k$ is not a $2$-element. Thus, $i(\sigma\tau^k)=1$.
\item Let $o(\sigma)\mid(q-1)$. Then $\sigma$ has no fixed places in $\cO_2$. Also, $\sigma$ fixes exactly two places in $\cO_1$, because the automorphism induced by $\sigma$ fixes exactly two places in $\bar\cO_1$; see \cite[Section 4]{GKT2006}. Therefore $i(\sigma)=2$.
The element $\sigma\tau^k$ is tame, fixes no places in $\cO_2$, and fixes exactly two places in $\cO_1$; thus, $i(\sigma\tau^k)=2$.
\item Let $o(\sigma)\mid(q+2q_0+1)$. Then $i(\sigma)=0$ since $o(\sigma)$ is coprime to the orders of $\cO_1$ and $\cO_2$. The same holds for $\sigma\tau^k$ and $i(\sigma\tau^k)=0$.

\item 
Since $q\geq8$, the quotient curve $\tcSq/\tSq$ is rational from \cite[Theorem 11.56]{HKT}.
From the Riemann-Hurwitz formula,
$$ (q^3+1)(q-2)=q^2(q^2+1)(q-1)(2\cdot0-2)+\Delta_{\tSq}. $$
Using the previous computation of $i(\sigma)$ for $o(\alpha)\nmid m$, we write
$$ \Delta_{\tSq} = (q^2+1)[(q-1)(q^2+1-mq)+(q^2-1)(m+1)]$$
$$ +\binom{q^2+1}{2}(q-2)\cdot2 + \frac{1}{4}q^2(q+2q_0+1)(q-1)(q-2q_0)\cdot \epsilon, $$
where $\epsilon=i(\alpha)$ for $1<o(\alpha)\mid m$; the elements of $\tSq$ have been counted as in \cite[Theorem 6.12]{GKT2006} with respect to their orders.
By direct checking, $\epsilon=0$.

Let $\sigma\in\tSq$ with $1<o(\sigma)\mid m$.
Let $P_1,\ldots,P_4\in\cSq$ be the fixed places of the induced automorphism $\bar\sigma\in\aut(\cSq)$, and $\Omega_1,\ldots,\Omega_4$ be the corresponding long orbits of $C_m$.
Let $P\in\Omega_1$; since $i(\sigma)=0$, $Q:=\sigma(P)\ne P$.
Since $\tau$ is regular on $\Omega_1$, there exists exactly one $j\in\{1,\ldots,m-1\}$ such that $\sigma\tau^j(P)=P$.
As $\sigma\tau^j$ commutes with $\tau$, $\sigma\tau^j$ fixes $\Omega_1$ pointwise.
Let $\bar N\leq\aut(\cSq)$ be the normalizer of $\langle\bar\sigma\rangle$ and $N\times C_m$ with $N\leq\tSq$ be the subgroup of $\aut(\tcSq)$ inducing $\bar N$.
Since $\bar N$ is transitive on $\{P_1,\ldots,P_4\}$, $N$ is transitive on $\{\Omega_1,\ldots,\Omega_4\}$. Hence, $\sigma\tau^j$ fixed $\Omega_i$ pointwise for any $i$.
Moreover, $\sigma\tau^k$ has no fixed places for $k\ne j$.
\end{itemize}
\end{proof}

\begin{remark}
If $\sigma\in\tSq$ and $\tau^j\in C_m$ satisfy $i(\sigma\tau^j)=4m$, then $\sigma$ and $\tau^k$ have the same order; otherwise, there would be a nontrivial power of $\sigma\tau^j$ fixing some places both in $\cO_1$ and in $\cO_2$, which is impossible.
\end{remark}

Given a subgroup $L\leq\aut(\tcSq)$, Theorem \ref{ContributionsSuzuki} provides a method to compute the genus of the quotient curve $\tcSq/L$.
We now compute the genus of $\tcSq/L$ whenever $L$ is a direct product $H\times C_{n}$, where $H$ is a subgroup of $\tSq$ and $C_{n}$ is a subgroup of $C_m$ of order $n$.

Note, from the classification of maximal subgroups of $S(q)$, that most subgroups of $\aut(\tcSq)$ have such decomposition into a direct product.
In fact, assume that the induced group $\bar L$ is either contained in a Singer group of order $q-2q_0+1$, nor is $\bar L$ a Suzuki subgroup $S(\hat q)$ with $q=\hat q^h$ and $h$ odd.
Then the orders of $\bar L$ and $C_n$ are coprime; hence, $L=H\times C_n$, where $H\leq\tSq$ is isomorphic to $\bar L$.
For instance, this is the case if the order of $L$ is not divisible by $\bar p^2$ for any prime divisor $\bar p$ of $q-2q_0+1$.

\begin{remark}\label{mprimo}
From {\rm \cite{GKT2006}} follows that the subgroups $H$ considered in Sections {\rm \ref{Sec:SuzQuot1}} - {\rm \ref{Sec:SuzQuot4}} are all the subgroups of $\tSq$; see {\rm \cite[Section 2, (2.1) - (2.10)]{GKT2006}}.
\end{remark}

\subsection{\bf $L=H\times C_n$ with $\bar H$ stabilizing an $\mathbb F_q$-rational place of $\cSq$}\label{Sec:SuzQuot1}

\begin{proposition}
Let $L\leq\aut(\tcSq)$ be a subgroup of order $r n$ with $r$ a divisor of $q-1$ and $n$ a divisor of $m$. Then
$$ g_L = \frac{1}{2}\cdot\frac{q-1}{r}\left(\frac{q^2+1}{n}-q-1\right). $$
\end{proposition}

\begin{proof}
Since $\gcd(r,n)=1$, we have $L=H\times C_n$ where $H\leq\tSq$ is cyclic of order $r$.
By the Riemann-Hurwitz formula, $(q^2+1)(q-2)= rn(2g_L-2)+\Delta_L $, where $ \Delta_L =(r-1)\cdot2+(n-1)(q^2+1) + (r-1)(n-1)\cdot2 $ from Theorem \ref{ContributionsSuzuki}. The claim follows.
\end{proof}

\begin{proposition}\label{2groups}
Let $L\leq\aut(\tcSq)$ be a subgroup of order $2^v n$, for some $v>0$ and $n$ a divisor of $m$. Let $2^u$ be the order of the (well-defined) subgroup of $L$ consisting of the involutions together with the identity. Then
$$g_L = \frac{ m(q^2+2q_0q-2^{u+1}q_0-2^v)-n(q^2-2^{v+1}+2^v) }{2^{v+1}n}.$$
\end{proposition}

\begin{proof}
Since $\gcd(2^v,n)=1$, $L=H\times C_n$ where $H\leq\tSq$ has order $2^v$.
By the Riemann-Hurwitz formula, $(q^2+1)(q-2)= 2^v n(2g_L-2)+\Delta_L $, where
$$ \Delta_L = (2^u-1)(m(2q_0+1)+1) + (2^v-2^u)(m+1) + (n-1)(q^2+1) + (2^v-1)(n-1)\cdot1 $$
from Theorem \ref{ContributionsSuzuki}; see \cite[Sec. 6, Type I]{GKT2006}. The claim follows by direct computation.
\end{proof}

\begin{corollary}
For $q_0=2^s$, let $u,v,n$ be integers such that $n\mid m$, $v\geq u\geq 0$, $v-u\leq s$, and $u\leq 2s+1$.
Suppose also that either $v\leq 2u$ and $(v-u)\mid(2s+1)$; or $v\leq u+\sqrt{2u+\frac{1}{4}}-\frac{1}{2}$.
Then there exists a group $L\leq\aut(\tcSq)$ such that $g_L$ is given by Proposition {\rm \ref{2groups}}.
\end{corollary}

\begin{proof}
If $v\leq 2u$ and $(v-u)\mid(2s+1)$, then the claim follows from \cite[Cor. 6.5]{GKT2006} (this is also the claim of \cite[Cor. 5.23]{BMXY}).
If $v\leq u+\sqrt{2u+\frac{1}{4}}-\frac{1}{2}$, then the claim follows from the  discussion after Corollary 5.20 in \cite[page 1367]{BMXY} (where $h_2=v-u$ and $h_3=u$, see also Corollary 4.4 (iii) in \cite{BMXY}).
Note that the condition $v\leq u+\sqrt{2u+\frac{1}{4}}-\frac{1}{2}$ is weaker than the condition $v\leq u+\log_2(u+1)$ given in \cite[Cor. 6.5]{GKT2006}.
\end{proof}

\begin{proposition}\label{normalizSuzuki}
Let $L\leq\aut(\tcSq)$ be a subgroup of order $2^v r n$ with $v,r>1$, $r\mid(q-1)$, and $n\mid m$. Let $2^u$ be the order of the (well-defined) subgroup of $L$ consisting of the involutions together with the identity. Then
$$ g_L = \frac{m[q^2+2q_0q-nq-2(n+2^u)q_0-n-2^v]+n(2^{v+1}-2^v+1)}{2^{v+1}rn}. $$
\end{proposition}

\begin{proof}
We have $L=H\times C_n$ where $H\leq\tSq$ fixes $P_\infty$. By the Riemann-Hurwitz formula, $(q^2+1)(q-2)=2^vrn(2g_L-2)+\Delta_L$, where
$$ \Delta_L = (2^u-1)(m(2q_0+1)+1) + (2^v-2^u)(m+1) + 2^v(r-1)\cdot2 $$
$$+ (n-1)(q^2+1) + 2^v(r-1)(n-1)\cdot2 + (2^v-1)(n-1)\cdot1 $$
from Theorem \ref{ContributionsSuzuki}; see \cite[Sec. 6, Type II]{GKT2006}. The claim follows by direct computation.
\end{proof}

\begin{corollary}
For $q_0=2^s$, let $u,v$ be integers such that $v\geq u\geq 0$, $v-u\leq s$, $u\leq 2s+1$, $r>1$ be a divisor of $q-1$ and $q/2^u-1$, and $n$ be a divisor of $m$.
Suppose also $v-u\leq\log_2(u+1)$ or $(v-u)\mid(2s+1)$.
Then there exists a group $L\leq\aut(\tcSq)$ such that $g_L$ is given by Proposition {\rm \ref{normalizSuzuki}}.
\end{corollary}

\begin{proof}
The corollary follows from \cite[Corollary 6.8]{GKT2006}.
\end{proof}

\begin{proposition}
Let $L\leq\aut(\tcSq)$ be a subgroup of order $2rn$ with $r>1$, $r\mid(q-1)$, and $n\mid m$. Then
$$ g_L = \frac{m[q^2+2q_0q-nq-(n+r+1)(2q_0+1)]+n(r+2)}{4rn}. $$
\end{proposition}

\begin{proof}
The group $\bar L$ is dihedral of order $2r$; see \cite[(2.6)]{GKT2006}.
Thus, $L=H\times C_n$, where $H\leq\tSq$ is dihedral of order $2r$.
By the Riemann-Hurwitz formula, $(q^2+1)(q-2)=2rn(2g_L-2)+\Delta_L$, where
$$ \Delta_L = (r-1)\cdot2 + r(m(2q_0+1)+1) + (n-1)(q^2+1) + r(n-1)\cdot1 + (r-1)(n-1)\cdot2 $$
from Theorem \ref{ContributionsSuzuki}. The claim follows by direct computation.
\end{proof}

\subsection{\bf $L=H\times C_n$ with $\bar H$ normalizing a Singer group of order $q+2q_0+1$}\label{Sec:SuzQuot2}

\begin{proposition}
Let $L\leq\aut(\tcSq)$ be a subgroup of order $rn$ with $r\mid(q+2q_0+1)$ and $n\mid m$. Then
$$ g_L = 1+\frac{q^2+1}{rn}\cdot\frac{q-1-n}{2}. $$
\end{proposition}

\begin{proof}
The group $\bar L$ is a cyclic Singer subgroup of order $r$; see \cite[(2.2)]{GKT2006}.
Since $r$ and $n$ are coprime, we have $L=H\times C_n$ where $H\leq\tSq$ is cyclic of order $r$.
By the Riemann-Hurwitz formula, $(q^2+1)(q-2)=rn(2g_L-2)+\Delta_L$, where $\Delta_L=(n-1)(q^2+1)$ from Theorem \ref{ContributionsSuzuki}. The claim follows.
\end{proof}

\begin{proposition}
Let $L\leq\aut(\tcSq)$ be a subgroup of order $2rn$ with $r\mid(q+2q_0+1)$ and $n\mid m$. Then
$$ g_L = 1+\frac{q^2+1}{rn}\cdot\frac{q-n-1}{4} - \frac{1}{4}\left[\frac{m}{n}(2q_0+1)+1\right]. $$
\end{proposition}

\begin{proof}
The group $\bar L$ is a dihedral group of order $2r$ containing a Singer subgroup of order $r$; see \cite[(2.7)]{GKT2006}.
Then $L=H\times C_n$, where $H\leq\tSq$ si dihedral of order $2r$.
By the Riemann-Hurwitz formula, $(q^2+1)(q-2)=2rn(2g_L-2)+\Delta_L$, where $\Delta_L=r(m(2q_0+1)+1)+(n-1)(q^2+1)+r(n-1)\cdot1$ from Theorem \ref{ContributionsSuzuki}. The claim follows by direct computation.
\end{proof}

\begin{proposition}
Let $L\leq\aut(\tcSq)$ be a subgroup of order $4rn$ with $r\mid(q+2q_0+1)$ and $n\mid m$. Then
$$ g_L = 1+\frac{q^2+1}{rn}\cdot\frac{q-n-1}{8}-\frac{1}{8}\left[\frac{m}{n}(2q_0+3)+3\right]. $$
\end{proposition}

\begin{proof}
The group $\bar L$ has order $4r$ and contains a Singer subgroup of order $r$ normalized by a cyclic group of order $4$; see \cite[(2.8)]{GKT2006}. Thus, $L=H\times C_n$ where $H\leq\tSq$ is isomorphic to $\bar L$.
By the Riemann-Hurwitz formula, $(q^2+1)(q-2)=4rn(2g_L-2)+\Delta_L$, where
$$\Delta_L=  r(m(2q_0+1)+1)+2r(m+1)+3r(n-1)\cdot1 + (n-1)(q^2+1) .$$
from Theorem \ref{ContributionsSuzuki}; see also \cite[Th. 6.11]{GKT2006}. The claim follows by direct computation.
\end{proof}

\subsection{\bf $L=H\times C_n$ with $\bar H$ normalizing a Singer group of order $q-2q_0+1$}\label{Sec:SuzQuot3}

\begin{proposition}
Let $L\leq\aut(\tcSq)$ be a direct product $H\times C_n$ of order $rn$, where $H\leq\tSq$ has order $r$, $C_n\leq C_m$ has order $n$, and $r,n$ are divisors of $m$. Then
$$ g_L = 1+\frac{m[q^2+(2q_0-n)q-2(n+1)q_0-n-4\gcd(r,n)+3]}{2rn}. $$
\end{proposition}

\begin{proof}
The group $\bar L\cong H$ is a cyclic Singer subgroup of order $r$; see \cite[(2.2)]{GKT2006}.
By the Riemann-Hurwitz formula, $(q^2+1)(q-2)=rn(2g_L-2)+\Delta_L$, where $\Delta_L = (\gcd(r,n)-1)4m + (n-1)(q^2+1)$ from Theorem \ref{ContributionsSuzuki}. The claim follows.
\end{proof}

\begin{proposition}
Let $L\leq\aut(\tcSq)$ be a direct product $H\times C_n$ of order $2rn$, where $H\leq\tSq$ has order $2r$, $C_n\leq C_m$ has order $n$, and $r,n$ are divisors of $m$. Then
$$g_L = \frac{q^3-(n+1)q^2+q-(2q_0+1)rm-4m(\gcd(r,n)-1)+3rn-n-1}{4rn}.$$
\end{proposition}

\begin{proof}
The group $\bar L\cong H$ is dihedral of order $2r$; see \cite[(2.7)]{GKT2006}.
By the Riemann-Hurwitz formula, $(q^2+1)(q-2)=2rn(2g_L-2)+\Delta_L$, where
$$\Delta_L = (\gcd(r,n)-1)4m + (n-1)(q^2+1) + r(m(2q_0+1)+1) + r(n-1)\cdot1$$
from Theorem \ref{ContributionsSuzuki}. The claim follows by direct computation.
\end{proof}

\begin{proposition}
Let $L\leq\aut(\tcSq)$ be a direct product $H\times C_n$ of order $4rn$, where $H\leq\tSq$ has order $4r$, $C_n\leq C_m$ has order $n$, and $r,n$ are divisors of $m$. Then
$$g_L = \frac{(q^2+1)(q-n-1)-m(2rq_0+3r-4+4\gcd(r,n))+5rn}{8rn}.$$
\end{proposition}

\begin{proof}
The group $\bar L\cong H$ has $r-1$ nontrivial elements of order a divisor of $r$, $2r$ involutions, and $2r$ elements of order $4$; see \cite[(2.7) and Type V.]{GKT2006}.
By the Riemann-Hurwitz formula, $(q^2+1)(q-2)=4rn(2g_L-2)+\Delta_L$, where
$$\Delta_L = (\gcd(r,n)-1)4m + (n-1)(q^2+1) + r(m(2q_0+1)+1) + 2r(m+1) + 3r(n-1)\cdot1$$
from Theorem \ref{ContributionsSuzuki}. The claim follows by direct computation.
\end{proof}

\subsection{\bf $L=H\times C_n$ with $\bar H$ a Suzuki subgroup of $\Sq$.}\label{Sec:SuzQuot4}
\ \\

For $q=2^{2s+1}$, let $\hat s\geq0$ be such that $2\hat s+1$ is a divisor of $2s+1$, and let $\hat q=2^{2\hat s+1}=2\hat{q}_0^2$. Let $h=\frac{2s+1}{2\hat s+1}$, so that $q=\hat{q}^h$.

\begin{lemma}\label{DivisibilitySuzuki}
For any odd integer $h$, the following holds.
\begin{itemize}
\item If $h\equiv1\pmod4$ with $\frac{h-1}{4}$ even, or $h\equiv3\pmod4$ with $\frac{h-3}{4}$ odd, then
$$ (\hat q-2\hat q_0+1)\mid(q-2q_0+1),\quad (\hat q+2\hat q_0+1)\mid(q+2q_0+1). $$
\item If $h\equiv1\pmod4$ with $\frac{h-1}{4}$ odd, or $h\equiv3\pmod4$ with $\frac{h-3}{4}$ even, then
$$ (\hat q-2\hat q_0+1)\mid(q+2q_0+1),\quad (\hat q+2\hat q_0+1)\mid(q-2q_0+1). $$
\end{itemize}
\end{lemma}

\begin{proof}
Define $d_+=\hat{q}+2\hat{q}_0+1$ and $d_-=\hat{q}-2\hat{q}_0+1$, which satisfy $\hat{q}^2\equiv-1\pmod{d_+}$ and $\hat{q}^2\equiv-1\pmod{d_-}$.

Assume $h\equiv1\pmod4$, $k=\frac{h-1}{4}$.
Then $q=\hat{q}^h\equiv\hat{q}\pmod{d_+,d_-}$.
Moreover, $q_0=\hat{q}^{2k}\hat{q}_0$, so that $q_0\equiv(-1)^k\hat{q}\pmod{d_+,d_-}$. Therefore,
$$ q\pm2q_0+1\equiv q\pm(-1)^k2\hat{q}_0+1\pmod{d_+,d_-}, $$
which yields the thesis for $h\equiv\pmod4$.

Assume $h\equiv3\pmod4$, $k=\frac{h-3}{4}$.
Then $q\equiv-\hat{q}\pmod{d_+,d_-}$. Moreover, $2q_0 = \hat{q}^{2k}\cdot 2\hat{q}_0\cdot\hat{q}$, so that
$$ 2q_0 \equiv (-1)^k 2\hat{q}_0(2\hat{q}_0-1) \equiv (-1)^k(2\hat{q}-2\hat{q}_0) \equiv (-1)^k(\hat{q}-1) \pmod{d_-}, $$
$$ 2q_0 \equiv (-1)^k 2\hat{q}_0(-2\hat{q}_0-1) \equiv (-1)^k(-2\hat{q}-2\hat{q}_0) \equiv (-1)^k(-\hat{q}+1) \pmod{d_+}. $$
Therefore, if $k$ is odd, then $q+2q_0+1\equiv0\pmod{d_+}$ and $q-2q_0+1\equiv0\pmod{d_-}$; if $k$ is even, then $q+2q_0+1\equiv0\pmod{d_-}$ and $q-2q_0+1\equiv0\pmod{d_+}$, which yields the thesis for $h\equiv3\pmod4$.
\end{proof}

From Lemma \ref{DivisibilitySuzuki} it is seen than the claim of Theorem 6.12 in \cite{GKT2006} does not hold in general:
the value of $\Delta$ given in \cite[Theorem 6.12]{GKT2006}, after changing the sign and dividing by $2$, is correct if and only if $1<(\tilde{q}-2\tilde{q}_0+1)\mid(q-2q_0+1)$.

\begin{proposition}
Let $L\leq\aut(\tcSq)$ be a direct product $H\times C_n$ of order $\hat q^2(\hat q^2+1)(\hat q-1)n$, where $H\leq\tSq$ is isomorphic to the Suzuki group $S(\hat q)$, $q=\hat q^h$, and $C_n\leq C_m$ has order $n$. Then
$$ g_L = 1 + \frac{(q^2+1)(q-2)-\Delta_L}{2n\hat q^2(\hat q^2+1)(\hat q-1)}, $$
where
$$ \Delta_L\! =\! (n-1)(q^2+1)\! +\! (\hat{q}^2+1)\!\left[ \hat{q}^2(\hat{q}-2)n\! +\! (\hat{q}-1)(q^2-mq+1+m\hat{q}+n\hat{q}+n) \right]+\delta, $$
with $\delta= \hat q^2(\hat q+2\hat q_0+1)(\hat q-1)(\gcd(\hat q-2\hat q_0+1,n)-1)m$ if $h\equiv1\pmod4$ with $\frac{h-1}{4}$ even or $h\equiv3\pmod4$ with $\frac{h-3}{4}$ odd, and $ \delta = \hat q^2(\hat q-2\hat q_0+1)(\hat q-1)(\gcd(\hat q+2\hat q_0+1,n)-1)m $
if $h\equiv1\pmod4$ with $\frac{h-1}{4}$ odd or $h\equiv3\pmod4$ with $\frac{h-3}{4}$ even.
\end{proposition}

\begin{proof}
By the Riemann-Hurwitz formula, $(q^2+1)(q-2)=\hat q^2(\hat q^2+1)(\hat q-1)n(2g_L-2)+\Delta_L$.
Computing the elements of $S(\hat q)$ as in \cite[Th. 6.12]{GKT2006}, we have that $H$ has $(\hat q^2+1)(\hat q-1)$ elements of order $2$; $(\hat q^2+1)(\hat q^2-\hat q)$ elements of order $4$; $\frac{1}{2}\hat q^2(\hat q^2+1)$ subgroups of order $\hat q-1$; $\frac{1}{4}\hat q^2(\hat q+2\hat q_0+1)(\hat q-1)$ subgroups of order $\hat q-2\hat q_0+1$; $\frac{1}{4}\hat q^2(\hat q-2\hat q_0+1)(\hat q-1)$ subgroups of order $\hat q+2\hat q_0+1$.
Therefore, from Theorem \ref{ContributionsSuzuki},
$$ \Delta_L = (\hat{q}^2+1)(\hat{q}-1)(m(2q_0+1)+1) + (\hat{q}^2+1)(\hat{q}^2-\hat{q})(m+1)+\frac{1}{2}\hat{q}^2(\hat{q}^2+1)(\hat{q}-2)\cdot2 $$
$$+ (n-1)(q^2+1) + (\hat{q}^2+1)(\hat{q}^2-1)(n-1)\cdot1 + \frac{1}{2}\hat{q}^2(\hat{q}^2+1)(\hat{q}-2)(n-1)\cdot2 + \delta, $$
where, from Lemma \ref{DivisibilitySuzuki},
$$ \delta = \frac{1}{4}\hat{q}^2(\hat{q}+2\hat{q}_0+1)(\hat{q}-1)(\gcd(\hat q-2\hat q_0+1,n)-1)\cdot4m $$
if $h\equiv1\pmod4$ with $\frac{h-1}{4}$ even or $h\equiv3\pmod4$ with $\frac{h-3}{4}$ odd, and
$$ \delta = \frac{1}{4}\hat q^2(\hat q-2\hat q_0+1)(\hat q-1)(\gcd(\hat q+2\hat q_0+1,n)-1)\cdot4m $$
if $h\equiv1\pmod4$ with $\frac{h-1}{4}$ odd or $h\equiv3\pmod4$ with $\frac{h-3}{4}$ even.
The claim follows by direct computation.
\end{proof}

\section{Galois subcovers of $\tcRq$}\label{Sec:QuotientsRee}

We keep the notations as in Sections \ref{Sec:PreliminaryResultsRee} and \ref{Sec:Ree}, and use the properties of $\aut(\tcRq)$ to obtain the genera of many Galois subcovers of $\tcRq$, in analogy with Section \ref{Sec:QuotientsSuzuki}.

Let $\varphi:\tcRq\to\cRq$ be the rational morphism $\varphi(x,y,z,t)=(x,y,z)$ and $\pi$ be the natural projection $\aut(\tcRq)\to\aut(\tcRq)/C_m\cong R(q)$, which satisfies $\varphi\circ h=\pi(h)\circ\varphi$ for any $h\in\aut(\tcRq)$. For $L\leq\aut(\tcRq)$, let $\bar L=\pi(L)=L/(L\cap C_m)$, $g_L$ be the genus of $\tcRq/L$, $g_{\bar L}$ be the genus of $\cRq/\bar L$, and $\Delta_L$ be the degree of the different divisor of the cover $\tcRq\to\tcRq/L$, which satisfies $(q^3+1)(q-2)=|L|(2g_L-2)+\Delta_L$.

We have $\Delta_L = \sum_{\sigma\in L,\sigma\ne id}i(\sigma)$, where $i(\sigma)$ is the sum, ranging over all places $P$ of $\tcRq$, of the number of non-negative integers $i$ such that $\sigma$ belongs to the $i$-th ramification group $L_P^{(i)}$ at $P$.
If $L$ is tame, $\Delta_L$ equals the sum, over all places $P$ of $\tcRq$, of the number of non-trivial elements of $L$ fixing $P$.

From Proposition \ref{OrbitsOfG}, $\aut(\tcRq)$ has exactly two short orbits $\cO_1$ and $\cO_2$ on $\tcRq$. The orbit $\cO_1$ is non-tame of size $q^3+1$, consisting of the $\mathbb F_q$-rational places of $\tcRq$; the orbit $\cO_2$ is tame of size $(q^3+1)q^3(q-1)$, consisting of the $\mathbb F_{q^6}\setminus\mathbb F_q$-rational places.

The non-tame and the tame short orbit of $\aut(\cRq)$ on $\cRq$ coincide respectively with the images $\bar\cO_1$ and $\bar\cO_2$ of $\cO_1$ and $\cO_2$ under $\varphi$. The places in $\bar\cO_1$ are totally ramified under $\cO_1$, while any other place of $\cRq$ is totally unramified under $\tcRq$.

The proofs of the following results are omitted.

\begin{proposition}
Let $L$ be a tame subgroup of $\aut(\tcRq)$ and $L_{C_m}:=L\cap C_m$.
Then
$$\Delta_L=(|L_{C_m}|-1)(q^3+1) +|L_{C_m}|n_1 + |L_{C_m}|n_2, $$
where
\begin{itemize}
\item $n_1$ counts the non-trivial relations $\bar h(\bar P)=\bar P$ with $\bar h\in\bar L$ when $\bar P$ varies in $\bar\cO_1$, namely
$$ n_1=\sum_{\bar h\in\bar L,\bar h\ne id} |\{\bar P\in\bar\cO_1 \mid \bar h(\bar P)=\bar P\}|; $$
\item  $n_2$ counts the non-trivial relations $\bar h(\bar P)=\bar P$ with $\bar h\in\bar L$ when $\bar P$ varies in $\bar \cO_2$, each counted with a multiplicity $l_{\bar h,\bar P}$ defined as the number of orbits of $\varphi^{-1}(\bar P)$ under the action of $L_{C_m}$ that are fixed by an element $h\in\pi^{-1}(\bar h)$. That is,
$$ n_2= \sum_{\bar h\in\bar L,\bar h\ne id} \sum_{\bar P\in\bar\cO_2,\bar h(\bar P)=\bar P} l_{\bar h,\bar P}. $$
\end{itemize}
\end{proposition}

\begin{proposition}
Let $L$ be a tame subgroup of $\aut(\tcRq)$. Assume that no non-trivial element in $\bar L$ fixes a place out of $\bar\cO_1$. Then
$$ g_L = g_{\bar L} + \frac{(q^3+1)(q-1)-|L_{C_m}|(q^3+3q_0q^2+q^2-q+3q_0-1)}{2|L|}. $$
\end{proposition}

\begin{proposition}
Let $L$ be a tame subgroup of $\aut(\tcRq)$ containing $C_m$. Then
$$ \Delta=(q-3q_0)(q^3+1) + (q-3q_0+1)\sum_{\bar h\in\bar L,\bar h\ne id}|\{\bar P\in \cRq(\mathbb{F}_{q^6})\mid \bar h(\bar P)=\bar P\}|. $$
\end{proposition}

\begin{proposition}\label{SameGenusRee}
Let $L$ be a subgroup of $\aut(\tcRq)$ containing $C_m$. Then $g_L = g_{\bar L}$.
\end{proposition}

We now compute the contributions $i(\sigma)$ to $\Delta_L$, starting with the description of the higher ramification groups.
Since $L_P^{(1)}$ is the Sylow $3$-subgroup of $L_P$, non-trivial higher ramification groups at $P$ only exist for $P$ in the non-tame orbit $\cO_1$. Up to conjugation, we assume that $P$ is the infinite place $P_\infty$ of $\tcRq$.

For $a,b,c,d\in\mathbb F_q$, $a\ne0$, let $\tilde\psi_{a,b,c,d}\in\aut(\tcRq)_{P_\infty}$ be the automorphism described in Section 
\ref{Sec:PreliminaryResultsRee}.
Let $\tilde\theta_{a,b,c,d}\in\aut(\tcRq)$ be the automorphism 
$$ \tilde\theta_{a,b,c,d}:(x,y,z,t)\mapsto(ax+b,a^{q_0+1}y+ab^{q_0}x+c,a^{2q_0+1}z-a^{q_0+1}b^{q_0}y+ab^{2q_0}x+d,at). $$
Then $\tRq_{P_\infty}=\{\tilde\theta_{a,b,c,d}\mid a,b,c,d\in\mathbb F_q,a\ne0\}$.

Let $T=\{\tilde\theta_{1,b,c,d}\mid b,c,d\in\mathbb F_q\}$ be the Sylow $3$-subgroup of $\tRq_{P_\infty}$.
Since $\tcRq$ has zero $3$-rank and $\tilde\theta_{1,b,c,d}$ is a $3$-element, $P_\infty$ is the unique fixed place of $\tilde\theta_{1,b,c,d}$.
From \cite[Proposition 2.3]{CO}, $T$ has exponent $9$; the derived subgroup $T^\prime$ of $T$ is elementary abelian of order $q^2$; the center $Z(T)$ of $T$ is contained in $T^\prime$, has order $q$, and coincides with the cubes of the elements of $T\setminus T^\prime$.

By direct checking, $\tilde\theta_{1,b,c,d}$ is in $T\setminus T^\prime$ if $b\ne0$, in $T^\prime\setminus Z(T)$ if $b=0$ and $c\ne0$, in $Z(T)$ if $b=c=0$.

\begin{proposition}
Let $G = \aut(\tcRq)$. Then
$$ G_{P_\infty}^{(0)} = \tRq_{P_\infty}\times C_m , $$
$$ G_{P_\infty}^{(1)} = \ldots G_{P_\infty}^{(m)} =T = \{\tilde\theta_{1,b,c,d}\mid b,c,d\in\mathbb F_q\}, $$
$$ G_{P_\infty}^{(m+1)} = \ldots G_{P_\infty}^{(m(3q_0+1))} =T^\prime = \{\tilde\theta_{1,0,c,d}\mid c,d\in\mathbb F_q\}, $$
$$ G_{P_\infty}^{(m(3q_0+1)+1)} = \ldots G_{P_\infty}^{(m(q+3q_0+1))} =Z(T)= \{\tilde\theta_{1,0,0,d}\mid d\in\mathbb F_q\}. $$
If $i>m(q+3q_0+1)=q^2-q+1$, then $G_{P_\infty}^{(i)}$ is trivial.
\end{proposition}

\begin{proof}
From the proof of Theorem 4.1 in \cite{Sk2016}, a local parameter at $P_\infty$ is $t/w_8$, where $w_8\in\mathbb F_3[x,y,z]$ is defined in \cite[Equations (3)]{Sk2016}.
Then the claim follows by the direct computation of the valuation at $P_\infty$ of $\tilde\theta_{1,b,c,d}(t/w_8)-t/w_8$.

An alternative proof is the following.
Given $d\in\mathbb F_q$, consider the group $L=\langle\tilde\theta_{1,0,0,d}\rangle\times C_m$.
From the Riemann-Hurwitz formula, $2g(\cRq)-2=3(2g_{\bar L}-2)+\Delta_{\bar L}$ and $2g(\tcRq)-2=3m(2g_{L}-2)+\Delta_{L}$.
From \cite[Theorem 3.1]{CO}, $\Delta_{\bar L}=2(q+3q_0+2)$.
Moreover, $\Delta_L=2i(\tilde\theta_{1,0,0,d})+(m-1)(q^3+1)+2(m-1)\cdot1$.
Since $g_{\bar L}=g_L$ by Proposition \ref{SameGenusRee}, we get $i(\tilde\theta_{1,0,0,d})=m(q+3q_0+1)+1$.
Similarly we obtain $i(\tilde\theta_{1,0,c,d})=m(3q_0+1)+1$ for $c\ne0$, and $i(\tilde\theta_{1,b,c,d})=m+1$ for $b\ne0$.
\end{proof}

Up to conjugacy, the maximal subgroups of $R(q)$ are the following; see \cite[Theorem 2.4]{CO} and \cite[Theorem A.14]{HKT}.
\begin{itemize}
\item The stabilizer of an $\mathbb F_q$-rational place of $\cRq$, of order $q^3(q-1)$.
\item The centralizer of an involution $\iota\in R(q)$ isomorphic to $\langle\iota\rangle\times{\rm PSL}(2,q)$, of order $q(q-1)(q+1)$.
\item The normalizer $N_+$ of a cyclic Singer group $\Sigma_+$ of order $q+3q_0+1$ acting semiregularly on $\cRq$; $N_+$ has order $6(q+3q_0+1)$.
\item The normalizer $N_-$ of a cyclic Singer group $\Sigma_-$ of order $q-3q_0+1$ which fixes six $\mathbb F_{q^6}$-rational places and acts semiregularly elsewhere; $N_-$ has order $6(q-3q_0+1)$.
\item The normalizer $N$ of a group of order $q+1$; $N$ has order $6(q+1)$.
\item The Ree group $R(q^\prime)$, where $q=(q^\prime)^h$ with $h$ prime.
\end{itemize}

For $\sigma\in\Rq$, the order $o(\sigma)$ of $\sigma$ satisfies exactly one of the following (see \cite{CO}):
$o(\sigma)=3$, or $o(\sigma)=9$, or $o(\sigma)=6$, or $o(\sigma)\mid(q-1)$, or $o(\sigma)\mid(q+1)$, or $o(\sigma)\mid(q+3q_0+1)$, or $o(\sigma)=q-3q_0+1$.

The proof of the following characterization is omitted as it is analogous to the one of Theorem \ref{ContributionsSuzuki}.

\begin{theorem}\label{ContributionsRee}
Let $\sigma\in\tRq\setminus\{id\}$ and $C_m=\langle\tau\rangle$. Denote by $o(\sigma)$ the order of $\sigma$. Then $i(\tau^k)=q^3+1$ for all $k=1,\ldots,m-1$ and one of the following cases occurs.
\begin{itemize}
\item $o(\sigma)=3$, $\sigma$ is in the center of a Sylow $3$-subgroup, $i(\sigma)=m(q+3q_0+1)+1=q^2-q+2$, and $i(\sigma\tau^k)=1$ for all $k=1,\ldots,m-1$;
\item $o(\sigma)=3$, $\sigma$ is not in the center of any Sylow $3$-subgroup, $i(\sigma)=m(3q_0+1)+1=q^2-q+2-mq$, and $i(\sigma\tau^k)=1$ for all $k=1,\ldots,m-1$;
\item $o(\sigma)=9$, $i(\sigma)=m+1$, and $i(\sigma\tau^k)=1$ for all $k=1,\ldots,m-1$;
\item $o(\sigma)=2$, $i(\sigma)=q+1$, and $i(\sigma\tau^k)=q+1$ for all $k=1,\ldots,m-1$;
\item $o(\sigma)=6$, $i(\sigma)=1$, and $i(\sigma\tau^k)=1$ for all $k=1,\ldots,m-1$;
\item $o(\sigma)\mid(q-1)$, $o(\sigma)\ne2$, $i(\sigma)=2$, and $i(\sigma\tau^k)=2$ for all $k=1,\ldots,m-1$;
\item $o(\sigma)\mid(q+1)$, $o(\sigma)\ne2$, $i(\sigma)=0$, and $i(\sigma\tau^k)=0$ for all $k=1,\ldots,m-1$;
\item $o(\sigma)\mid(q+3q_0+1)$, $i(\sigma)=0$, and $i(\sigma\tau^k)=0$ for all $k=1,\ldots,m-1$;
\item $o(\sigma)\mid(q-3q_0+1)$, $i(\sigma)=0$, $i(\sigma\tau^j)=6m$ for exactly one $j\in\{1,\ldots,m-1\}$, and $i(\sigma\tau^k)=0$ for all $k\in\{1,\ldots,m-1\}\setminus\{j\}$.
\end{itemize}
\end{theorem}

We now compute the genus of $\tcRq/L$ whenever $L\leq\aut(\tcRq)$ is a direct product $H\times C_n$, where $H\leq\tRq$ and $C_n\leq C_m$.
For instance, this happens for all $L\leq\aut(\tcRq)$ for which $\bar L$ is not contained in a Singer subgroup of order $q-3q_0+1$ and is not a Ree subgroup $R(q^\prime)$.

\begin{remark}
From {\rm \cite[Theorem A.14]{HKT}} follows that the subgroups $H$ considered in Sections {\rm \ref{Sec:ReeQuot1}} - {\rm \ref{Sec:ReeQuot6}} are all the subgroups of $\tRq$.
\end{remark}

\subsection{\bf $L=H\times C_n$ with $\bar H$ stabilizing and $\mathbb F_q$-rational place of $\cRq$}\label{Sec:ReeQuot1}

\begin{proposition}\label{Normalizer3Group}
Let $L\leq\aut(\tcRq)$ be a subgroup of order $3^w r n$, for some $w\geq0$, $r\mid(q-1)$, and $n\mid m$.
For $T$ a Sylow $3$-subgroup of $L$, let $3^u$ be the order of the center $Z(T)$ and $3^v$ be the order of the derived subgroup $T^\prime$.
Then
$$ g_L=\frac{q^4-(n+1)q^3-(3^v-1)q^2+(m(3^v-3^u)+3^v)q-3^w (m-n)+3^v(m-1)+\epsilon}{2\cdot 3^w r n}, $$
where $\epsilon=0$ if $r$ is odd, and $\epsilon= 3^{w-v+u} n(q+3^{v-u})$ if $r$ is even.
\end{proposition}

\begin{proof}
Since $3^w$ and $n$ are coprime, $L=H\times C_n$ where $H\leq\tSq$ and $C_n\leq C_m$. By the Riemann-Hurwitz formula, $(q^3+1)(q-2)=3^w rn(2g_L-2)+\Delta_L$.
From Theorem \ref{ContributionsRee},
$$\Delta_L=(3^u-1)(q^2-q+2)+(3^{v}-3^u)(q^2-q+2-mq)+(3^w-3^v)(m+1)$$
$$ + (n-1)(q^3+1) + (3^w-1)(n-1)\cdot1 + \delta,$$
with
$$\delta=\begin{cases}(r-2)3^w n\cdot2+3^{w-v+u}(3^{v-u}-1)n\cdot1+3^{w-v+u}n(q+1) &\textrm{if} \ r \ \textrm{is even}, \\
(r-1)3^w n\cdot2&\textrm{if} \ r \ \textrm{is odd}; \\ \end{cases}$$
see \cite[Proposition 5.2]{CO2} for the value of $\delta$.
The claim follows by direct computation.
\end{proof}

The following corollary is implied by \cite[Corollary 5.4]{CO}.

\begin{corollary}
For $q_0=3^s$, let $u,v,w,r,n$ be integers such that $w=v\geq u\geq0$, $u\leq2s+1$, $v-u\leq2s+1$, $r\mid\gcd(3^{2s+1}-1,3^{u}-1,3^{v-u}-1)$, and $n\mid m$. Then there exists a group $L\leq\aut(\tcRq)$ such that $g_L$ is given by Proposition {\rm \ref{Normalizer3Group}}.
\end{corollary}

Another numerical sufficient condition for the existence of $L$ as in Proposition \ref{Normalizer3Group} is given in \cite[Example 7.4]{BMXY}.

\begin{corollary}
For $q_0=3^s$, let $r$ be a divisor of $q-1$, and define $a$ as the smallest positive integer such that $3^a\equiv1\pmod r$. Let $\ell$ be a divisor of $\frac{2s+1}{a}$, and $t$ be an integer such that $0\leq t\leq \frac{2s+1}{a}-\ell$. Define $t^\prime=\lceil t/\ell\rceil-1$. Let $h$ be an integer such that $(t^\prime+2)\ell\leq h\leq \frac{2s+1}{a}$.
Define $u=a h$, $v=a(h+t)$, $w=a(2\ell-t)$.
Then there exists a group $L\leq\aut(\tcRq)$ such that $g_L$ is given by Proposition {\rm \ref{Normalizer3Group}}.
\end{corollary}

\subsection{\bf $L=H\times C_n$ with $\bar H$ centralizing an involution $\iota\in\Rq$}\label{Sec:ReeQuot2}
\ \\

Let $\bar H\leq\langle\iota\rangle\times{\rm PSL}(2,q)$.
Then either $\bar H\leq{\rm PSL}(2,q)$ or $\bar H=\langle\iota\rangle\times(\bar H\cap{\rm PSL}(2,q))$.
Since the subgroups of ${\rm PSL}(2,q)$ are known (\cite[Hauptsatz 8.27]{Huppert}), we classify $L$ in terms of $\bar H\cap{\rm PSL}(2,q)$; $j$ will denote the index of $\bar H\cap{\rm PSL}(2,q)$ in $\bar H$.

\begin{proposition}
Let $L\leq\aut(\tcRq)$ have order $j 3^v n$, where $j\in\{1,2\}$, $n\mid m$, and $\bar H\cap{\rm PSL}(2,q)$ is elementary abelian of order $3^v$, with $v\leq2s+1$.
Then
$$ g_L= \frac{q^4-(n+1)q^3-(3^v-1)q^2+[m(3^v-1)+3^v-n(j-1)]q+3^v(jn-1)}{2j 3^v n}. $$
\end{proposition}

\begin{proof}
From \cite[Theorem 4.9]{CO} and Theorem \ref{ContributionsRee}, the $3$-elements $\sigma$ of $L$ satisfy $i(\sigma)=m(3q_0+1)+1$.
By the Riemann-Hurwitz formula, $(q^3+1)(q-2)=j3^v n(2g_L-2)+\Delta_L$, where
$$ \Delta_L = (3^v-1)(m(3q_0+1)+1) + (j-1)(q+1) + (j-1)(3^v-1)\cdot1 + (n-1)(q^3+1)$$
$$ + (3^v-1)(n-1)\cdot1 +(j-1)(n-1)(q+1) + (j-1)(3^v-1)(n-1)\cdot1 $$
from Theorem \ref{ContributionsRee}. The claim follows.
\end{proof}

\begin{proposition}
Let $L\leq\aut(\tcRq)$ have order $jrn$, where $j\in\{1,2\}$, $n\mid m$, and $\bar H\cap{\rm PSL}(2,q)$ is cyclic of order a divisor $r$ of $\frac{q+1}{2}$. Then
$$g_L = 1 + \frac{q+1}{2r}\left(\frac{(q^2-q+1)(q-1)}{jn}-\frac{q^2-q}{j}-\gcd(r,2)\right). $$
\end{proposition}

\begin{proof}
By the Riemann-Hurwitz formula and Theorem \ref{ContributionsRee}, $(q^3+1)(q-2)=jrn(2g_L-2)+\Delta_L$, where
$$ \Delta_L = (j-1)(q+1) + (n-1)(q^3+1) + (j-1)(n-1)(q+1)$$
if $r$ is odd, and
$$ \Delta_L = (j-1)(q+1) + 1\cdot(q+1) + (n-1)(q^3+1) + (j-1)\cdot1\cdot(q+1) $$
$$ + (j-1)(n-1)(q+1) + 1\cdot(n-1)(q+1) + (j-1)\cdot1\cdot(n-1)(q+1)$$
if $r$ is even. The claim follows by direct computation.
\end{proof}

\begin{proposition}
Let $L\leq\aut(\tcRq)$ have order $jrn$, where $j\in\{1,2\}$, $n\mid m$, and $\bar H\cap{\rm PSL}(2,q)$ is cyclic of order a divisor $r$ of $\frac{q-1}{2}$. Then
$$g_L = \frac{q-1}{2r}\left(\frac{q^3+1}{jn}-\frac{q^2+q}{j}-1\right). $$
\end{proposition}

\begin{proof}
By the Riemann-Hurwitz formula and Theorem \ref{ContributionsRee},
$(q^3+1)(q-2)=jrn(2g_L-2)+\Delta_L$, where
$$ \Delta_L = (j-1)(q+1)+(r-1)\cdot2+(n-1)(q^3+1)+(j-1)(r-1)\cdot2$$
$$ + (j-1)(n-1)(q+1) + (r-1)(n-1)\cdot2 + (j-1)(r-1)(n-1)\cdot2. $$
\end{proof}

\begin{proposition}
Let $L\leq\aut(\tcRq)$ have order $2jrn$, where $j\in\{1,2\}$, $n\mid m$, and $\bar H\cap{\rm PSL}(2,q)$ is dihedral of order $2r$ for some divisor $r$ of $\frac{q+1}{2}$. Then
$$g_L = 1+\frac{q+1}{2r}\left( \frac{q-1}{2}\cdot\frac{q^2-(n+1)q+1}{jn}-\frac{r+\gcd(r,2)}{2} \right). $$
\end{proposition}

\begin{proof}
The claim follows from $(q^3+1)(q-2)=2jrn(2g_L-2)+\Delta_L$, where
$$ \Delta_L= (j-1)(q+1)+r(q+1)+(n-1)(q^3+1)+(j-1)r(q+1)$$
$$+(j-1)(n-1)(q+1)+r(n-1)(q+1)+(j-1)r(n-1)(q+1) $$
if $r$ is odd,
$$ \Delta_L= (j-1)(q+1)+1\cdot(q+1)+r(q+1)+(n-1)(q^3+1)+(j-1)\cdot1\cdot(q+1)$$
$$+(j-1)r(q+1)+(j-1)(n-1)(q+1)+1\cdot(n-1)(q+1)$$
$$ +r(n-1)(q+1)+(j-1)\cdot1(n-1)(q+1)+(j-1)r(n-1)(q+1)  $$
if $r$ is even.
\end{proof}

\begin{proposition}
Let $L\leq\aut(\tcRq)$ have order $2jrn$, where $j\in\{1,2\}$, $n\mid m$, and $\bar H\cap{\rm PSL}(2,q)$ is dihedral of order $2r$ for some divisor $r$ of $\frac{q-1}{2}$. Then
$$ g_L= \frac{q^2-1}{4jr}\left(\frac{q^2-q+1}{n}-q\right)-\frac{(r+1)(q-1)}{4r}. $$
\end{proposition}

\begin{proof}
The claim follows from $(q^3+1)(q-2)=2jrn(2g_L-2)+\Delta_L$, where
$$ \Delta_L = (j-1)(q+1)+(r-1)\cdot2+r(q+1)+(n-1)(q^3+1)$$
$$+(j-1)(r-1)\cdot2+(j-1)r(q+1)+(j-1)(n-1)(q+1)+(r-1)(n-1)\cdot2$$
$$+r(n-1)(q+1)+(j-1)(r-1)(n-1)\cdot2+(j-1)r(n-1)(q+1). $$
\end{proof}

\begin{proposition}
Let $L\leq\aut(\tcRq)$ have order $12jn$, where $j\in\{1,2\}$, $n\mid m$, and $\bar H\cap{\rm PSL}(2,q)$ is isomorphic to the alternating group $A_4$. Then
$$g_L =1+\frac{1}{24j}\left[\frac{m}{n}(q^2-1)(q+3q_0)-4j(q+3)+\frac{m}{n}(q^2-9)-24\left(\frac{m}{n}q_0+3\right)+8\left(9-q\frac{q^2-1}{8}\right)\right]. $$
\end{proposition}

\begin{proof}
The claim follows from $(q^3+1)(q-2)=2jrn(2g_L-2)+\Delta_L$, where
$$ \Delta_L = (j-1)(q+1)+ 3(q+1)+8(m(3q_0+1)+1)+(n-1)(q^3+1) $$
$$ +(j-1)\cdot3\cdot(q+1) + (j-1)\cdot8\cdot1 + (j-1)(n-1)(q+1) +3(n-1)(q+1) $$
$$ + 8(n-1)\cdot1 + (j-1)\cdot3\cdot(n-1)(q+1) + (j-1)\cdot8\cdot(n-1)\cdot1. $$
\end{proof}

\begin{proposition}
Let $L\leq\aut(\tcRq)$ have order $j3^v r n$, where $j\in\{1,2\}$, $v\leq2s+1$, $r\mid\frac{q-1}{2}$, $r\mid(3^v-1)$, $n\mid m$, and $\bar H\cap{\rm PSL}(2,q)$ is the semidirect product of an abelian $3$-subgroup of order $3^v$ with a cyclic group of order $r$. Then
$$g_L =1 + \frac{ q^4-(n+1)q^3 - (3^v-1)q^2 + (3^v m + 3^v - jn - m + n)q - 3^v(2jrn - jn - 2rn + 4r + 2n -3)}{2 j 3^v r n}. $$
\end{proposition}

\begin{proof}
The claim follows from $(q^3+1)(q-2)=j3^vrn(2g_L-2)+\Delta_L$, where
$$ \Delta_L =  (j-1)(q+1)+(3^v-1)(q^2-q+2-mq)+3^v(r-1)2+(n-1)(q^3+1)$$
$$+(j-1)(3^v-1)+(j-1)3^v(r-1)2 + (j-1)(n-1)(q+1) + (3^v-1)(n-1)1$$
$$ + 3^v(r-1)2 + (j-1)(3^v-1)(n-1)1 + (j-1)3^v(r-1)(n-1)2. $$
\end{proof}

\begin{proposition}
Let $L\leq\aut(\tcRq)$ have order $\frac{1}{2}j(\hat{q}+1)\hat{q}(\hat{q}-1) n$, where $j\in\{1,2\}$, $q=\hat{q}^h$ for some $h$, $n\mid m$, and $\bar H\cap{\rm PSL}(2,q)$ is isomorphic to ${\rm PSL}(2,\hat{q})$. Then
$$g_L=1+ \frac{ q^4-(n+1)q^3-(\hat{q}^2-1)q^2 - [\hat{q}^2(\frac{1}{2}nj-n-1)-\frac{1}{2}\hat{q}nj+n(j-1)+m]q }{j(\hat{q}+1)\hat{q}(\hat{q}-1) n}$$
$$ -\frac{ \frac{1}{2}\hat{q}^2nj(\hat{q}+1)+\hat{q}-2nj }{j(\hat{q}+1)(\hat{q}-1) n}. $$
\end{proposition}

\begin{proof}
The claim follows from $(q^3+1)(q-2)=\frac{1}{2}j(\hat{q}+1)\hat{q}(\hat{q}-1)(2g_L-2)+\Delta_L$, where
$$ \Delta_L = (\hat{q}-1)(\hat{q}+1)(q^2-q+2-mq) + \frac{\hat{q}\left(\hat{q}+1\right)}{2}\left(\frac{\hat{q}-1}{2} - 1\right)2 + 
    \frac{\hat{q}(\hat{q}-1)}{2}( q+1 )$$
$$+ (n-1)(q^3+1)+(j-1)(q+1) + (j-1)(\hat{q}-1)(\hat{q}+1) + 
    (j-1)\frac{\hat{q}(\hat{q}+1)}{2}\left(\frac{\hat{q}-1}{2} - 1\right)2$$
$$+ (j-1)\frac{\hat{q}(\hat{q}-1)}{2}( q+1 ) + 
    (j-1)(n-1)(q+1) + (\hat{q}^2-1)(n-1)  $$
$$+ \frac{\hat{q}(\hat{q}+1)}{2}\left(\frac{\hat{q}-1}{2} - 1\right)(n-1)2 + \frac{\hat{q}(\hat{q}-1)}{2}(n-1)( q+1 ) + (j-1)(n-1)(\hat{q}^2-1) $$
$$+(j-1)\frac{\hat{q}(\hat{q}+1)}{2}\left(\frac{\hat{q}-1}{2} - 1\right)(n-1)2+ (j-1)\frac{\hat{q}(\hat{q}-1)}{2}(n-1)( q+1 ).$$
\end{proof}

The study of groups $L=H\times C_n$ with $\bar H$ centralizing an involution of $\Rq$ is now complete; see \cite[Theorem 4.11]{CO}.

\subsection{\bf $L=H\times C_n$ with $\bar H$ normalizing a Singer group of order $q+3q_0+1$}\label{Sec:ReeQuot3}
\ \\

If $\bar H$ normalizes a Singer subgroup of $\Rq$ of order $q+3q_0+1$, then $H=A\rtimes B$ where $A$ is cyclic of order a divisor $r$ of $q+3q_0+1$ and $B$ is cyclic of order a divisor of $6$; see \cite[Proposition 4.13]{CO}.
Moreover, the elements $\sigma\in H$ of order $3$ satisfy $i(\sigma)=m(3q_0+1)+1$; see \cite[Theorem 4.14]{CO}.

\begin{proposition}
Let $L\leq\aut(\tcRq)$ have order $rn$, where $r\mid(q+3q_0+1)$ and $n\mid m$. Then
$$g_L= 1 + \frac{q+1}{2}\cdot\frac{q^2-q+1}{rn}(q-n-1). $$
\end{proposition}

\begin{proof}
From the Riemann-Hurwitz formula and Theorem \ref{ContributionsRee}, $(q^3+1)(q-2)=rn(2g_L-2)+\Delta_L$, where $\Delta_L = (n-1)(q^3+1)$.
\end{proof}

\begin{proposition}
Let $L\leq\aut(\tcRq)$ have order $2rn$, where $r\mid(q+3q_0+1)$ and $n\mid m$. Then
$$g_L= 1 + \frac{q+1}{4}\left(\frac{q^2-q+1}{rn}(q-n-1)-1\right). $$
\end{proposition}

\begin{proof}
From the Riemann-Hurwitz formula and Theorem \ref{ContributionsRee}, $(q^3+1)(q-2)=2rn(2g_L-2)+\Delta_L$, where $\Delta_L = r(q+1) + (n-1)(q^3+1) + r(n-1)(q+1)$.
\end{proof}

\begin{proposition}
Let $L\leq\aut(\tcRq)$ have order $3rn$, where $r\mid(q+3q_0+1)$ and $n\mid m$. Then
$$g_L= 1 + \frac{q^4-(n+1)q^3-2rq^2+[2r(m+1)+1]q-(2r+1)(n+1)}{6rn}. $$
\end{proposition}

\begin{proof}
The claim follows from $(q^3+1)(q-2)=3rn(2g_L-2)+\Delta_L$, where
$$\Delta_L = 2r(q^2-q+2-mq) + (n-1)(q^3+1) + 2r(n-1)\cdot1.$$
\end{proof}

\begin{proposition}
Let $L\leq\aut(\tcRq)$ have order $6rn$, where $r\mid(q+3q_0+1)$ and $n\mid m$. Then
$$g_L= 1 + \frac{(q^3+1)(q-2)-r(2q^2-(2m-n+2)q+5n+2)}{12rn}. $$
\end{proposition}

\begin{proof}
The claim follows from $(q^3+1)(q-2)=6rn(2g_L-2)+\Delta_L$, where
$$\Delta_L = r[2(q^2-q+2-mq)+1\cdot(q+1)+2\cdot1] + (n-1)(q^3+1) + r(n-1)[2\cdot1+1\cdot(q+1)+2\cdot1].$$
\end{proof}

\subsection{\bf $L=H\times C_n$ with $\bar H$ normalizing a Singer group of order $q-3q_0+1$}\label{Sec:ReeQuot4}
\ \\

If $\bar H$ normalizes a Singer subgroup of $\Rq$ of order $q-3q_0+1$, then $H=A\rtimes B$ where $A$ is cyclic of order a divisor $r$ of $q+3q_0+1$ and $B$ is cyclic of order a divisor of $6$; see \cite[Proposition 4.13]{CO}.
Moreover, the elements $\sigma\in H$ of order $3$ satisfy $i(\sigma)=m(3q_0+1)+1$; see \cite[Theorem 4.17]{CO}.

\begin{proposition}
Let $L=H\times C_n\leq\aut(\tcRq)$ with $H\leq \tRq$ of order $r$ and $C_n\leq C_m$ of order $n$, for some divisors $r,n$ of $m$.
Then
$$g_L= 1 + \frac{(q^3+1)(q-n-1)-6(\gcd(r,n)-1)m}{2rn}. $$
\end{proposition}

\begin{proof}
From the Riemann-Hurwitz formula and Theorem \ref{ContributionsRee}, $(q^3+1)(q-2)=rn(2g_L-2)+\Delta_L$, where $\Delta_L = (\gcd(r,n)-1)6m + (n-1)(q^3+1)$.
\end{proof}

\begin{proposition}
Let $L=H\times C_n\leq\aut(\tcRq)$ with $H\leq \tRq$ of order $2r$ and $C_n\leq C_m$ of order $n$, for some divisors $r,n$ of $m$.
Then
$$g_L= 1 + \frac{(q^3+1)(q-n-1)-6(\gcd(r,n)-1)m-rn(q+1)}{4rn}. $$
\end{proposition}

\begin{proof}
From the Riemann-Hurwitz formula and Theorem \ref{ContributionsRee}, $(q^3+1)(q-2)=2rn(2g_L-2)+\Delta_L$, where
$$\Delta_L = (\gcd(r,n)-1)6m + r(q+1) + (n-1)(q^3+1) + r(n-1)(q+1).$$
\end{proof}

\begin{proposition}
Let $L=H\times C_n\leq\aut(\tcRq)$ with $H\leq \tRq$ of order $3r$ and $C_n\leq C_m$ of order $n$, for some divisors $r,n$ of $m$.
Then
$$g_L= 1 + \frac{(q^3+1)(q-n-1)-6(\gcd(r,n)-1)m - 2r(q^2-q+n+1-mq)}{6rn}. $$
\end{proposition}

\begin{proof}
The claim follows from $(q^3+1)(q-2)=3rn(2g_L-2)+\Delta_L$, where
$$\Delta_L = (\gcd(r,n)-1)6m + 2r(q^2-q+2-mq) + (n-1)(q^3+1)+ 2r(n-1)\cdot1.$$
\end{proof}

\begin{proposition}
Let $L=H\times C_n\leq\aut(\tcRq)$ with $H\leq \tRq$ of order $6r$ and $C_n\leq C_m$ of order $n$, for some divisors $r,n$ of $m$.
Then
$$g_L= 1 + \frac{(q^3+1)(q-n-1)-6(\gcd(r,n)-1)m  - r[2q^2-(2m-n+2)q+5n+2]}{12rn}. $$
\end{proposition}

\begin{proof}
The claim follows from $(q^3+1)(q-2)=6rn(2g_L-2)+\Delta_L$, where
$$\Delta_L = (\gcd(r,n)-1)6m + r[2(q^2-q+2-mq)+1\cdot(q+1)+2\cdot1]+(n-1)(q^3+1)+r(n-1)[2\cdot1+1\cdot(q+1)+2\cdot1].$$
\end{proof}

\subsection{\bf $L=H\times C_n$ with $\bar H$ normalizing a group of order $q+1$}\label{Sec:ReeQuot5}
\ \\

From \cite[Lemma 4.18]{CO}, $\bar H\subseteq N$ where $N=(E\times D)\rtimes C_3$, with $E$ elementary abelian of order $4$, $D$ dihedral of order $\frac{q+1}{2}$, and $C_3$ of order $3$.
Let $E_{\bar H}=E\cap\bar H$ and $D_{\bar H}=D\cap\bar H$.

Assume $3\mid|\bar H|$.
From \cite[Lemma 4.23]{CO}, $\bar H=(E_{\bar E}\times D_{\bar H})\rtimes\langle\sigma\rangle$, where $\sigma$ has order $3$.
If $|D_{\bar H}|$ is odd, then every element in $\bar H\setminus(E_{\bar H}\times D_{\bar H})$ has order $3$.
If $|D_{\bar H}|$ is even, $\bar H$ has exactly $|E_{\bar H}||D_{\bar H}|$ elements of order $3$ and $|E_{\bar H}||D_{\bar H}|$ elements of order $6$.
From \cite[Proposition 4.25]{CO}, every element $\sigma\in\bar H$ of order $3$ satisfies $i(\sigma)=m(3q_0+1)+1$.

From \cite[Theorem 4.24]{CO}, one of the following cases holds:
either $\bar H\leq E\times D$; or $\bar H=E\times D_{\bar H}\rtimes\langle\sigma\rangle$; or $\bar H=D_{\bar H}\rtimes\langle\sigma\rangle$, with $o(\sigma)=3$.

\begin{proposition}
Let $L\leq\aut(\tcRq)$ have order $ijrn$ with $r\mid\frac{q+1}{4}$, $i\in\{1,2,4\}$, $i=|E_{\bar H}|$, $j\in\{1,2\}$, $jr=|D_{\bar H}|$, $n\mid m$. Then
$$g_L=  1 + \frac{(q+1)\left[(q^2-q+1)(q-n-1)-n(i(j-1)r+i-1)\right]}{2ijrn}. $$
\end{proposition}

\begin{proof}
From the Riemann-Hurwitz formula and Theorem \ref{ContributionsRee}, $(q^3+1)(q-2)=ijrn(2g_L-2)+\Delta_L$, where
$$\Delta_L = \left[i-1+i(j-1)r\right](q+1)+(n-1)(q^3+1)+\left[i-1+i(j-1)r\right](n-1)(q+1).$$
\end{proof}

\begin{proposition}
Let $L\leq\aut(\tcRq)$ have order $12jrn$ with $r\mid\frac{q+1}{4}$, $j\in\{1,2\}$, $jr=|D_{\bar H}|$, $n\mid m$. Then
$$g_L= \frac{(q^3+1)(q-n-1)-n(q+1)[3+4(j-1)r]-8rm(3q_0+1)+16jrn}{24jrn}. $$
\end{proposition}

\begin{proof}
From the Riemann-Hurwitz formula and Theorem \ref{ContributionsRee}, $(q^3+1)(q-2)=12jrn(2g_L-2)+\Delta_L$, where
$$\Delta_L = \left[3+4(j-1)r\right](q+1) + 8r(m(3q_0+1)+1) + (j-1)8r\cdot1$$
$$ + (n-1)(q^3+1)+ \left[3+4(j-1)r\right](n-1)(q+1) + 8jr(n-1)\cdot1.$$
\end{proof}

\begin{proposition}
Let $L\leq\aut(\tcRq)$ have order $3jrn$ with $r\mid\frac{q+1}{4}$, $j\in\{1,2\}$, $jr=|D_{\bar H}|$, $n\mid m$. Then
$$g_L= 1 + \frac{(q^3+1)(q-n-1)-n(q+1)(j-1)r-2rm(3q_0+1)-2jrn}{6jrn}.$$
\end{proposition}

\begin{proof}
$$\Delta_L = (j-1)r(q+1) + 2r(m(3q_0+1)+1) + (j-1)2r\cdot1$$
$$ + (n-1)(q^3+1)+ (j-1)r(n-1)(q+1) + 2jr(n-1)\cdot1.$$
\end{proof}

\subsection{\bf $L=H\times C_n$ with $\bar H$ a Ree subgroup of $\Rq$}\label{Sec:ReeQuot6}
\ \\

Let $\hat s\geq0$ be such that $2\hat s+1$ divides $2s+1$, $h=\frac{2s+1}{2\hat s+1}$, and $q=\hat q^h=3\hat q_0^2$.

\begin{proposition}
Let $L=H\times C_n\leq\aut(\tcRq)$ have order $\hat{q}^3(\hat{q}^3+1)(\hat{q}-1)n$, where $H\leq\tRq$ is isomorphic to the Ree group $R(\hat{q})$ and $C_n\leq C_m$ has order $n$. Then
$$ g_L = 1 + \frac{(q^3+1)(q-2)-\Delta_L}{2n\hat{q}^3(\hat{q}^3+1)(\hat{q}-1)}, $$
where $\Delta_L$ is given in Equation \eqref{Valori},
with $\delta = \hat{q}^3(\hat{q}-1)(\hat{q}+1)(\hat{q}+3\hat{q}_0+1)(\gcd(\hat{q}-3\hat{q}_0+1,n)-1)m$ if $h\equiv1\pmod6$ with $\frac{h-1}{6}$ even or $h\equiv5\pmod6$ with $\frac{h-5}{6}$ odd; 
$\delta = \hat{q}^3(\hat{q}-1)(\hat{q}+1)(\hat{q}-3\hat{q}_0+1)(\gcd(\hat{q}+3\hat{q}_0+1,n)-1)m$ if $h\equiv1\pmod6$ with $\frac{h-1}{6}$ odd or $h\equiv5\pmod6$ with $\frac{h-5}{6}$ even.
\end{proposition}

\begin{proof}
The number of nontrivial elements of $R(\hat{q})$ can be computed with respect to their order as follows; see \cite[Section 4.5]{CO}. The group $R(\hat{q})$ has $\hat{q}^2(\hat{q}^2-\hat{q}+1)$ involutions; $(\hat{q}^3+1)(\hat{q}-1)$ elements of order $3$ lying in the center of some Sylow $3$-subgroup of $R(\hat{q})$; $(\hat{q}^3+1)(\hat{q}^2-\hat{q})$ elements of order $3$ not lying in the center of any Sylow $3$-subgroup of $R(\hat{q})$; $(\hat{q}^3+1)(\hat{q}^3-\hat{q}^2)$ elements of order $9$; $\hat{q}^2(\hat{q}^2-\hat{q}+1)(\hat{q}+1)(\hat{q}-1)$ elements of order $6$; $\binom{\hat{q}^3+1}{2}(\hat{q}-3)$ elements of order a divisor of $\hat{q}-1$ different from $2$; $\frac{1}{6}\hat{q}^3(\hat{q}-1)(\hat{q}+1)(\hat{q}+3\hat{q}_0+1)(\hat{q}-3\hat{q}_0)$ nontrivial elements of order a divisor of $\hat{q}-3\hat{q}_0+1$; $\frac{1}{6}\hat{q}^3(\hat{q}-1)(\hat{q}+1)(\hat{q}-3\hat{q}_0+1)(\hat{q}+3\hat{q}_0)$ nontrivial elements of order a divisor of $\hat{q}+3\hat{q}_0+1$; the remaining nontrivial elements have order a divisor of $\frac{q+1}{4}$.

From the Riemann-Hurwitz formula, $(q^3+1)(q-2)=\hat{q}^3(\hat{q}^3+1)(\hat{q}-1)n(2g_L-2)+\Delta_L$. Since $(\hat{q}-1)\mid(q-1)$ and $(\hat{q}+1)\mid(q+1)$, Theorem \ref{ContributionsRee} yields
$$ \Delta_L = \hat{q}^2(\hat{q}^2-\hat{q}+1)n(q+1)+ (\hat{q}^3+1)(\hat{q}-1)[1\cdot(q^2-q+2)+(n-1)\cdot1]$$
$$ + (\hat{q}^3+1)(\hat{q}^2-\hat{q})[1\cdot(q^2-q+2-mq)+(n-1)\cdot1] + (\hat{q}^3+1)(\hat{q}^3-\hat{q}^2)[1\cdot(m+1)+(n-1)\cdot1]$$
\begin{equation}\label{Valori} + \hat{q}^2(\hat{q}^2-\hat{q}+1)(\hat{q}+1)(\hat{q}-1)n\cdot1 + \binom{\hat{q}^3+1}{2}(\hat{q}-3)n\cdot2 + (n-1)(q^3+1) + \delta, \end{equation}
where
$$\delta = \sum_{\sigma\tau \in I}i(\sigma\tau),\quad\textrm{with}\quad I=\big\{(\sigma\tau\in L \ :\ \sigma\in H\setminus\{id\},\ \tau\in C_n, \ o(\sigma)\mid(\hat q\pm3\hat q_0+1)\big\}. $$
From \cite[Lemmas 4.34, 4.35]{CO} and Theorem \ref{ContributionsRee}, the following holds.
\begin{itemize}
\item If $h\equiv3\pmod6$, then $(\hat q\pm3\hat q_0+1)\mid(q+1)$; thus, $\delta=0$.
\item If $h\equiv1\pmod6$ with $\frac{h-1}{6}$ even or $h\equiv5\pmod6$ with $\frac{h-5}{6}$ odd, then $(\hat{q}-3\hat{q}_0+1)\mid(q-3q_0+1)$ and $(\hat{q}+3\hat{q}_0+1)\mid(q+3q_0+1)$; thus,
$$ \delta= \frac{1}{6}\hat{q}^3(\hat{q}-1)(\hat{q}+1)(\hat{q}+3\hat{q}_0+1)(\gcd(\hat{q}-3\hat{q}_0+1,n)-1)\cdot6m. $$
\item If $h\equiv1\pmod6$ with $\frac{h-1}{6}$ odd or $h\equiv5\pmod6$ with $\frac{h-5}{6}$ even, then $(\hat{q}-3\hat{q}_0+1)\mid(q+3q_0+1)$ and $(\hat{q}-3\hat{q}_0+1)\mid(q-3q_0+1)$; thus,
$$ \delta = \frac{1}{6}\hat{q}^3(\hat{q}-1)(\hat{q}+1)(\hat{q}-3\hat{q}_0+1)(\gcd(\hat{q}+3\hat{q}_0+1,n)-1)\cdot6m. $$
\end{itemize}
The claim follows by direct computation.
\end{proof}

\section{The cases $q=2^3$, $q=2^5$, and $q=3^3$}\label{Sec:NewGenera}

The results of Sections \ref{Sec:QuotientsSuzuki} and \ref{Sec:QuotientsRee} provide many new genera for $\mathbb F_{q^4}$-maximal curves in characteristic $2$ and for $\mathbb F_{q^6}$-maximal curves in characteristic $3$.
To exemplify this fact, Table \ref{tabbellabbella} shows genera of $\mathbb F_{8^4}$-, $\mathbb F_{8^5}$-, and $\mathbb F_{27^6}$-maximal curves which are new, up to our knowledge; that is, they are new with respect to the genera provided by \cite{ABB,AQ,BMXY,CO,CO2,DO,DO2,FG,GOS,GSX}.

\begin{table}[htbp]
\begin{small}
\caption{New genera for maximal curves}\label{tabbellabbella}
\begin{center}
\begin{tabular}{c|c}
$F$ & new genera for $F$-maximal curves \\
\hline\hline
$\mathbb F_{2^{12}}$ & $13,19,45,196$ \\
\hline $\mathbb F_{2^{20}}$ & 77, 86, 106, 125, 146, 205, 247, 314, 324,376, 422, 447, 526, 616, 650, 735, \\
& 856, 906, 1322, 1482, 1824, 1874, 2666, 3076, 3760, 3810, 7632, 15376\\
\hline $\mathbb F_{3^{18}}$ & 337, 347, 445, 455, 675, 694, 891, 910, 1075, 1429, 1431, 1459, 1469,\\
&  2125, 2154, 2862, 2866, 2919, 2938, 4254, 4381, 4387, 4471, 4501, 4511,\\
& 4725, 5825, 6651, 8775, 8781, 8787, 8946, 9003, 9022, 9457, 9463, 10217,\\
& 11654, 12951, 13507, 13597, 13627, 17575, 18927, 20438, 27027, 27198, \\
& 27255, 30745, 35151, 40885, 40975, 61503, 81783, 81954 \\
\end{tabular}
\end{center}
\end{small}
\end{table}

\end{document}